\documentclass{article}
\usepackage[utf8]{inputenc}
\usepackage[margin=1in]{geometry}

\usepackage{todonotes}
\usepackage{acronym}
\usepackage{xcolor}
\usepackage{amsmath}
\usepackage{amssymb}
\usepackage{mathtools}
\usepackage{subcaption}
\usepackage{hyperref}
\hypersetup{colorlinks=true, linkcolor=black}
\usepackage[ruled]{algorithm2e}
\usepackage{mathrsfs}
\definecolor{ccolor}{RGB}{203,96,21}

\newcommand{\R}{\mathbb{R}}

\newcommand{\N}{\mathbb{N}}

\newcommand{\Lie}{\mathcal{L}}

\newcommand{\A}{\mathcal{A}}

\newcommand{\cs}{\mathcal{C}}

\DeclarePairedDelimiter{\ceil}{\lceil}{\rceil}

\DeclarePairedDelimiterX{\inp}[2]{\langle}{\rangle}{#1, #2}

\DeclarePairedDelimiter{\Mp}{\mathcal{M}_+(}{)}

\newcommand{\old}[1]{{\color{black} #1}}

\usepackage{amsthm}
\newtheorem{thm}{Theorem}[section]
\newtheorem{lem}[thm]{Lemma}
\newtheorem{prop}[thm]{Proposition}
\newtheorem{prob}[thm]{Problem}

\newtheorem{rmk}{Remark} 

\newcommand{\bbmu}{\boldsymbol{\mu}}
\newcommand{\bell}{\boldsymbol{\ell}}

\usepackage{nicefrac}
\newlength{\exfiglength}
\usepackage{pgfplots}
\pgfplotsset{compat=1.10}
\usetikzlibrary{calc}
\usetikzlibrary {arrows.meta, positioning}
\usepackage{xcolor}
\definecolor{myMagenta}{rgb}{1,0,1}
\renewcommand\footnotemark{}

\title{\bf Peak Time-Windowed Risk Estimation \\ of Stochastic Processes
}

\author{Jared Miller$^1$,  Niklas Schmid$^1$, Matteo Tacchi$^2$, Didier Henrion$^3$,  Roy S. Smith$^1$
\thanks{$^1$J.\ Miller, N.\ Schmid, and R.S.\ Smith are with the Automatic Control Laboratory (IfA) and NCCR Automation, Department of Information Technology and Electrical Engineering (D-ITET), ETH Z\"{u}rich, Physikstrasse 3, 8092, Z\"{u}rich, Switzerland ( \{jarmiller, rsmith\}@control.ee.ethz.ch).}
\thanks{$^2$ M.\ Tacchi is with Univ. Grenoble Alpes, CNRS, Grenoble INP (Institute of Engineering Univ. Grenoble Alpes), GIPSA-lab, 38000 Grenoble, France. (e-mail: matteo.tacchi@gipsa-lab.fr)}
\thanks{$^3$ D.\ Henrion is with the Polynomial Optimization group of LAAS-CNRS, Universit\'e de Toulouse, CNRS, Toulouse, France; and the 
Faculty of Electrical Engineering, Czech Technical University in Prague, Czech Republic. (henrion@laas.fr)}
\thanks{J.\ Miller and R.S.\ Smith are  partially supported by the Swiss National Science Foundation under NCCR Automation, grant agreement 51NF40\_180545. 
}
}

\begin{document}

\maketitle
\thispagestyle{empty}
\pagestyle{empty}

%%%%%%%%%%%%%%%%%%%%%%%%%%%%%%%%%%%%%%%%%%%%%%%%%%%%%%%%%%%%%%
\begin{abstract}
\label{sec:abstract}
% Instances of peak estimation include bounding the 
% Estimating extreme values of state functions (peak estimation) has applications in 
This paper develops a method to upper-bound extreme-values of time-windowed risks for stochastic processes. Examples of such risks include the maximum average or 90\% quantile of the current along a transmission line in any 5-minute window. This work casts the time-windowed risk analysis problem as an infinite-dimensional linear program in occupation measures. In particular, we employ the coherent risk measures of the mean and the expected shortfall (conditional value at risk) to define the maximal time-windowed risk along trajectories. The infinite-dimensional linear program must then be truncated into finite-dimensional optimization problems, such as by using the moment-sum of squares hierarchy of semidefinite programs. The infinite-dimensional linear program will have the same optimal value as the original nonconvex risk estimation task under compactness and regularity assumptions, and the sequence of semidefinite programs will converge to the true value under additional properties of algebraic characterization. The scheme is demonstrated for risk analysis of example stochastic processes.

% \textbf{Color code: }\old{Old text from CDC paper} New text in TAC paper \urg{Comments that require attention}

% \urg{Maybe for CDC+LCSS? But the deadline is so close (March 1). We can go straight to a journal like TAC, I'd prefer not to rewrite this paper between conference and journal.}

\end{abstract}
\section{Introduction}
\label{sec:introduction}

This work analyzes risks associated with state-function values of a stochastic process over a specified time window. A motivating example in power systems includes the maximum  current being generated by a synchronous machine \cite{machowski1997power}. When the synchronous machine starts to spin, an \textit{inrush current} of $\sim$5-6 times the rated current capacity of the machine can be generated. Such exceedance is tolerable, because the maximum instantaneous current only occurs for a short time range \cite{turner2010transformer}.  A time-windowed risk such as the mean, quantile, or \ac{ES} may give an operator a better picture of the safety of the power system, where the duration of this time window is related to the heat dissipation properties of the relevant components. 
% While the instantaneous current can be high for short periods of time (such as in inertial response of synchronous machines), power operators are interested in analyzing and regulating the maximum average current draw over time windows. \urg{Cite this} 

The problem setting discussed in this paper will involve stochastic processes described by a generator $\Lie$ for trajectories evolving in a state domain $X \subset \R^n$ (the stochastic trajectory $x(t)$ `follows' the generator $\Lie$). In the context of power systems, such stochastic processes could originate from thermal noise or the intermittencies of demand and photovoltaic/wind generation.
Trajectories will evolve starting from an initial set $X_0 \subseteq X$, and will stop upon first contact with the boundary $\partial X$ (with an exit time distribution $\tau_X = \inf\{ t : X_t \in \partial X\}$) or upon reaching a finite time horizon of $T$. Given a risk measure $\mathcal{R}$ (e.g., mean, 90\%-quantile), a state function $p$,  and a time window $h \in [0, T]$, the time-windowed risk estimation problem is as follows:

\begin{prob}
    \label{prob:risk_window}
Find a stopping time $t^*$ and initial condition $x_0^*$ to supremize 
\begin{subequations}
\label{eq:risk_window}
\begin{align}
    P^*= & \sup_{t^*, \ x_0^*}  \quad \mathcal{R}\left(\frac{1}{h}\int_{t^*-h}^{t^*} p(x(t')) dt'\right) \label{eq:risk_window_obj} \\
     \textrm{s.t.} \quad & x(t) \ \text{follows} \  \Lie \quad  \forall t \in [0, \min(t^*, \tau_X)] 
\label{eq:risk_window_dynamics}\\
     & x(0) = x_0^* \\
     & t^* \in [h, T], \ x_0^* \in X_0 \label{eq:risk_window_time}
\end{align}
\end{subequations}
\end{prob}

The optimization variables of Problem \ref{prob:risk_window} are the stopping time $t^*$ and the initial condition $x_0^*$. The parenthesized quantity in \eqref{eq:risk_window_obj} is %probability distribution 
a random variable
formed by the %pushforward 
evaluation
of $p$ along the stochastic trajectories $x(t)$, as averaged in the time window $[t^*-h, t^*]$. In the continuous-time case, this It\^{o}-integrated average acts as a normalized marginalization which collapses state-value data $p(x(t'))$ in times $t' \in [t^*-h, t^*]$ into a single probability distribution. In the case of a discrete-time stochastic process with time-increment $\Delta t$, the integral in \eqref{eq:risk_window_obj} should be interpreted as the time-windowed average taken over $h/\Delta t$ time steps. 

When the time window $h$ approaches 0, Problem \eqref{prob:risk_window} can be interpreted as an instantaneous maximal-risk estimation problem \cite{miller2023chancepeak}. Conversely, as $h$ approaches $T$ (with constraint \eqref{eq:risk_window_time} restricting $t^*$ to $T$), Problem \ref{prob:risk_window} assumes the form of a risk-sensitive Lagrange-type stochastic optimal control problem \cite{bonalli2023first}. If $\mathcal{R}$ is additionally chosen to be the mean with $h \rightarrow T$, then Problem  \eqref{eq:risk_window_time} will approach a standard expectation-form stochastic control problem \cite{aastrom2012introduction}.

Problem \ref{prob:risk_window} is a generically nonconvex finite-dimensional optimization in terms of its variables $(t^*, x_0^*)$. In this work, we will use measure-theoretic methods in optimal control \cite{lewis1980relaxation, altman2021constrained} to lift Problem \ref{prob:risk_window} into a convex infinite-dimensional \ac{LP} when $\mathcal{R}$ is a coherent risk measure \cite{delbaen2000coherent}. This \ac{LP} will be posed in terms of the following variables:
\begin{itemize}
    \item An initial measure
    \item A stopping (terminal) measure
    \item A pair of temporally split occupation measures
    \item An optional measure or a set of finite-dimensional terms to implement the absolute-continuity representation of coherent risk measures when $\mathcal{R}$ is not the mean.
\end{itemize}
The time-windowing in the cost is implemented by adding an additional constant state $s$ (taking on the value of the stopping time). The scaling of discretizations of the derived \acp{LP} will increase based on $(n+2)$, where the factor of `$2$' represents the original time $t$ and the augmented time $s$.

Infinite-dimensional \acp{LP} are a standard approach in the analysis and control of stochastic systems. 
Problem instances include optimal control \cite{lewis1980relaxation, altman2021constrained}, maximal mean estimation \cite{cho2002linear, kashima2010optimization}, barrier functions for safety \cite{prajna2004stochastic, prajna2007framework}, reach-avoid verification \cite{xue2023reach}, instantaneous risk estimation \cite{miller2023chancepeak}, long-time averaging \cite{tobasco2018optimal}, probabilistic planning \cite{jasour2019risk}, and exit-time location \cite{henrion2021moment}. Infinite-dimensional \ac{LP} methods for non-stochastic processes include reachable set estimation/backwards-reachable-set maximizing control \cite{Henrion_2014, majumdar2014convex, kariotoglou2013approximate, schmid2022probabilistic}, global attractor estimation \cite{goluskin2020attractor, schlosser2021converging}, maximum positively invariant set estimation \cite{ oustry2019inner}, robust analysis of uncertain systems \cite{miller2021uncertain}, analysis and control of hybrid systems \cite{prajna2004safety, zhao2019optimal, miller2023hybrid}, and analysis of time-delay systems \cite{prajna2005methods, miller2023delay}.

Infinite-dimensional \acp{LP} must be truncated into finite-dimensional optimization problems in order to admit numerical solutions. Truncation methods include the moment-\ac{SOS} hierarchy of \acp{SDP} (in the case of polynomial structure) \cite{lasserre2009moments, fantuzzi2020bounding}, random sampling with guarantees \cite{MOHAJERINESFAHANI201643}, successively refined gridded discretization \cite{cho2002linear, altman2021constrained}, and neural network certificates with SAT-based verification \cite{abate2021fossil}. All of these methods are vulnerable to a curse of dimensionality as the number of states $n$ and the complexity of representation (resp. polynomial degree, number of samples, grid spacing, number of layers) increases. Grid-based approaches for constrained risk-aware stochastic analysis and control include interpreting the \ac{ES} as a 
robust disturbance over the trajectory history \cite{chow2015risk}, 
minimizing the maximal \ac{ES} \cite{chapman2022optimizing}, using log-sum-exp constraints to approximate the \ac{ES} \cite{chapman2021risk}, and nested risk costs that admit a solution by dynamic programming in the unconstrained setting \cite{ruszczynski2010risk}. 
% \urg{Niklas: help out with approximate dynamic programming? Don't know if we want to touch on joint chance constraints.} 
We separately note the recent work of \cite{bonalli2023first} which characterized a Pontryagin Maximum Principle for continuous-time risk-aware stochastic control.

% The work in \ac{LP} in \cite{lewis1980relaxation} 

% Instances of successful application of infinite-dimensional \ac{LP} methods in dynamical systems anal

% an initial probability measure, a terminal (stopping measure), a pair of temporally split occupation measures, and a set of optional measures/variables dependent on the risk functional $\mathcal{R}$ used.

% \urg{I feel like there are some nasty characterizations of this integral. The explanation here could definitely be improved.}

% \urg{Prior work on stochastic processes. Safety, barrier functions, reach-avoid, etc.}

% \urg{Our prior work on stochastic analysis. Chance-Peak, Unsafe Prob. Peak estimation more generally.}

% \urg{Introduction goes here. }

% \urg{Literature review here.}

The contributions of this work include
% \urg{Contributions of this work are,
\begin{itemize}
    \item A presentation of the time-windowed risk estimation problem.
    \item A reformulation of time-windowed risk estimation into primal-dual \acp{LP} in occupation measures, with specific focus posed on the case where $\mathcal{R}$ is the mean or the \ac{ES}.
    \item Proofs of no relaxation gap and convergence of discretizations under compactness and regularity assumptions.    
    % \item Extensions to multiple choices of risk measures (mean, \ac{ES}, mean+variance, mean+standard deviation).
\end{itemize}

This paper has the following structure: 
Section \ref{sec:preliminaries}  reviews preliminaries such as notation, risk measures, stochastic processes, and occupation measures. 
Section \ref{sec:risk_setup} presents an example of instantaneous vs. time-windowed risk estimation, details assumptions posed on Problem \ref{prob:risk_window}, and derives augmented temporal support sets for later use in the \acp{LP}.
Section \ref{sec:mean_lp} presents measure \ac{LP} formulations for Problem \ref{prob:risk_window} when $\mathcal{R}$ is chosen to be the mean, 
and finds conditions for no-relaxation-gap and strong duality. Section \ref{sec:cvar_lp} modifies the \acp{LP} to allow for the \ac{ES} as a risk measure in the time-windowing risk analysis task.
Section \ref{sec:risk_lmi} truncates the measure \ac{LP} using the moment-\ac{SOS} hierarchy and accounts for computational complexity. Section \ref{sec:examples} details examples demonstrating the effectiveness of this approach in bounding the time-windowed risk. 
% Section \ref{sec:extensions} lists extensions to the measure \ac{LP} method, such as in incorporating other uncertainties in dynamics, and formulating the optimal control problem of minimizing the maximal time-windowed average cost. 
Section \ref{sec:conclusion} concludes the paper.

% Appendix \ref{app:duality} proves strong duality between measure and continuous function \ac{LP} formulations for time-windowed risk estimation. Appendix \ref{app:risk_measures} documents ways that several. 
% Appendix \ref{app:convergence} proves that increasing the polynomial degree of the moment-\ac{SOS} hierarchy approximations to the measure \ac{LP} will result in asymptotic convergence from above to the true peak time-windowed risk value.

% \urg{Fill in the paper structure}
% Section \ref{sec:preliminaries} will review preliminaries such as notation, notions of stability for linear systems, and \ac{SOS} proofs of polynomial nonnegativity. Section \ref{sec:full_method} will present 
% The paper is concluded in Section \ref{sec:conclusion}.
\section{Preliminaries}
\label{sec:preliminaries}

\subsection{Acronyms/Initialisms}
\begin{acronym}
\acro{BSA}{Basic Semialgebraic}
% \acro{DDC}{Data Driven Control}

% \acro{GAS}{Globally Asymptotically Stable}

\acro{ES}{Expected Shortfall}
\acroindefinite{ES}{an}{a}

\acro{LMI}{Linear Matrix Inequality}
\acroplural{LMI}[LMIs]{Linear Matrix Inequalities}
\acroindefinite{LMI}{an}{a}

% \acro{LQR}{Linear Quadratic Regulator}
% \acroplural{LMI}[LMIs]{Linear Matrix Inequalities}
% \acroindefinite{LQR}{an}{a}

\acro{LP}{Linear Program}
\acroindefinite{LP}{an}{a}
% \acro{OCP}{Optimal Control Problem}

\acro{ODE}{Ordinary Differential Equation}
\acroindefinite{ODE}{an}{a}
% \acro{POP}{Polynomial Optimization Problem}

\acro{PSD}{Positive Semidefinite}

% \acro{PD}{Positive Definite}

% \acro{PDE}{Partial Differential Equation}

\acro{SDE}{Stochastic Differential Equation}
\acroindefinite{SDE}{an}{a}

\acro{SDP}{Semidefinite Program}
\acroindefinite{SDP}{an}{a}

\acro{SOS}{Sum of Squares}
\acroindefinite{SOS}{an}{a}

% \acro{WSOS}{Weighted Sum of Squares}

\end{acronym}

\subsection{Notation}
% \urg{Fill in the notation}

The symbol $\R^n$ and $\N^n$ denote the $n$-dimensional sets of real numbers and natural numbers respectively. The notation $[\cdot]_+$ will refer to the scalar function $[z]_+ = \max(z, 0)$. The scalar identity function $\textrm{id}_{\R}$ is defined as $\forall q \in \R: \textrm{id}_{\R}(q) = q.$

Given a vector $x \in \R^n$ and a multi-index $\alpha \in \N^n$, the monomial $\prod_{i=1}^n x_i^{\alpha_i}$ will be represented as $x^\alpha$.
The set of polynomials in an indeterminate $x$ with real coefficients is $\R[x]$. The degree of a polynomial $p \in \R[x]$ is $\deg p$. The set of polynomials with degree at most $k$ is $\R[x]_{\leq k}$.

\subsection{Measure Theory}

Consider a compact set $K \in \R^n$. The space of continuous functions from $K$ to $\R$ is $C(K)$. This space is equipped with a supremum-norm $\|f\|_{C(K)} := \max\{|f(x)| \; | \; x \in K\}$ , giving $C(K)$ a Banach structure \cite{tao2011introduction}.
The dual space $C(K)'$ can be identified as the space of continuous linear functionals from $C(K)$ to $\R$. For elements $f \in C(K)$ and $\mu \in C(K)'$, the duality notation $\inp{f}{\mu} = \int f \; d\mu \in \R$ will be used to  denote their pairing. The subset of $k$-times differentiable functions over a set $K$ is $C^k(K).$ If the set $K$ can be partitioned into a two-argument set $x = (x_1, x_2)$ (where the dimensions of $x_1$ and $x_2$ sum to $n$), then the set $C^{k_1, k_2}(K)$ is the set of continuous functions defined over $K$ that are $k_1$-times continuously differentiable in $x_1$ and are $k_2$-times continuously differentiable in $x_2$.

The nonnegative subcones of $C(K)$ and $C(K)'$ are $C_+(K)$ and $C_+(K)'$ respectively, with the definitions:
\begin{subequations}
\begin{align}
C_+(K) &= \{f \in C(K) \; | \; \forall x \in K, \, f(x) \geq 0 \} \\
    C_+(K)' &= \{ \mu \in C(K)' \; | \; \forall f \in C_+(K), \, \inp{f}{\mu} \geq 0\}. \label{eq:riesz_markov}
\end{align}
\end{subequations}
The expression of $C_+(K)'$ in \eqref{eq:riesz_markov} is the Riesz-Markov characterization of Radon (nonnegative) measures supported on $K$, with notation $\Mp{K} = C_+(K)'$. Throughout the rest of this paper, the notation $\Mp{K}$ will be used instead of $C_+(K)'$, reflecting the more natural measure interpretation.

Given a measure $\mu \in \Mp{K}$, the measure of a set $A \subseteq K$ is $\mu(A).$ The mass of $\mu$ may be equivalently expressed as $\mu(K) = \inp{1}{\mu}.$ The measure $\mu$ is a \textit{probability measure} if $\inp{1}{\mu} = 1$. The Dirac delta $\delta_{x=x'}$ is a probability distribution supported solely at the point $x'$, satisfying $\forall f \in C(X): \inp{f}{\delta_{x = x'}} = f(x').$ 

Let $L \in \R^m$ be a compact set.
Given measures $\mu \in \Mp{K}, \nu \in \Mp{L}$, the product measure $\mu \otimes \nu$ is the unique measure satisfying the relation $\forall f \in C(K), g \in C(L): \ \inp{f(x)\cdot g(y)}{\mu(x)\otimes\nu(y)} = (\inp{f}{\mu})  (\inp{g}{\nu})$. For every linear operator $\mathscr{L} : C(K) \rightarrow C(L)$, there exists a unique adjoint $\mathscr{L}^\dagger: C(L)' \rightarrow C(K)'$ satisfying $\forall f \in C(K), \nu \in C(L)': \ \inp{\mathscr{L} f}{\nu} = \inp{f}{\mathscr{L}^\dagger \nu}$. An example of such a linear operator is differentiation (through the convention of partial derivatives of Radon measures), for which the following expression may be derived (if $f \in C^1(K)$) via integration by parts:
    \begin{align}
        \mathscr{L} f = \nicefrac{\partial}{\partial x}, \qquad  \mathscr{L}^\dagger = - \nicefrac{\partial}{\partial x},\\
        \intertext{to maintain the relation}
         \inp{v}{\nicefrac{\partial \mu}{\partial x_i}} = - \inp{\nicefrac{\partial v}{\partial x_i}}{\mu}.
    \end{align}

Given a map $q: K \rightarrow L$ and a measure $\nu \in \Mp{L},$ the pushforward of $\nu$ along $q$ (denoted as $q_\# \mu \in \Mp{K}$) is the unique measure satisfying $\forall g \in C(K): \inp{g}{q_\# \mu} = \inp{g \circ q}{\mu}.$ Letting $(x_1, x_2) \in K_1 \times K_2$ denote a point in a product space, the projection $\pi^{x_1}$ onto the first coordinate is $\pi^{x_1} (x_1, x_2) = x_1$. The pushforward of the projection operator $\pi_\#^{x_1}$ is a marginalization operator, preserving the term $x_1$ and marginalizing out $x_2$.

% $\forall f \in C(Y): \ \inp{f(y)}{Q_\# \mu(y)} = \inp{f(Q(\rw{s}))}{\mu(\rw{s})}
Let $\mu, \zeta \in \Mp{K}$ be a pair of nonnegative measures supported in $K$. The measure $\zeta$ is \textit{absolutely continuous} to $\mu$ ($\zeta \ll \mu$) if $\forall A \subseteq K$: $\mu(A)=0$ implies that $\zeta(A)=0$. If $\zeta \ll \mu$, then there exists a unique function $g \in C_+(K)$ such that $\forall f \in C(K): \int_{K} f(x) d \zeta(x) = \int_{K} f(x) g(x) d \mu(x)$. The function $g(x)$ is called the \textit{Radon-Nikodym derivative} of $\zeta$ w.r.t. $\mu$, and may be expressed as $g(x) = \frac{d \zeta}{  d \mu}(x)$.
The measure $\zeta$ is \textit{dominated} by $\mu$ ($\zeta \leq \mu$) if $\forall A \subseteq K: \mu(A) \geq \zeta(A).$ Note that domination is a stronger condition than absolute continuity, because domination imposes that $ \forall x \in K: \ \frac{d \zeta}{  d \mu}(x) \leq 1$.

\subsection{Risk Measures}

%Let $\mathbb{E}$ denote the expectation operator, and let $V: (\Omega,\mathbb{P}) \rightarrow \mathbb{R}$ be a univariate random variable with finite first and second moments ($\abs{\mathbb{E}[V]} < \infty$, $\mathbb{E}[V^2] < \infty$).
Let $(\Omega, \mathbb{P})$ be a probability space. A \textit{risk measure} is a mapping from the space of real valued random variables $V: (\Omega,\mathbb{P}) \rightarrow \mathbb{R}$ to the extended set of real numbers $\R \cup \infty$. A risk measure $\rho$ is a \textit{coherent risk measure} (in a maximum convention) if the following properties are satisfied (for all possible random variables $V_1, V_2$) \cite{artzner1999coherent}:
\begin{align*}
    &\textrm{Normalized}  & & \rho(0) = 0 \\
   & \textrm{Monotonicity}  & & \mathbb{P}\left(V_1 \leq V_2\right)=1 \implies \rho(V_1) \geq \rho(V_2) \\
   & \textrm{Sub-additivity}  & & \rho(V_1+V_2) \leq \rho(V_1) + \rho(V_2) \\
   & \textrm{Positive Homogeneity}   & & \forall \alpha \geq 0 : \ \rho(\alpha V_1) = \alpha \rho(V_1) \\
  & \textrm{Translation invariance} & &  \forall c \in \R: \ \rho(V_1 + c) = \rho(V_1) + c
\end{align*}

Every coherent risk measure has an interpretation in terms of absolute continuity. Specifically, let $\psi$ represent the probability density of the random variable $V$ in the real line $\R$. For every coherent risk measure $\rho$, there exists a function $F:\mathbb{R} \rightarrow \mathbb{R}$ and a constraint set $C$ such that \cite{delbaen2000draft}
\begin{align}
    \rho(V) = \sup_{\zeta \ll \psi, \zeta \in C} \inp{F}{\zeta}. \label{eq:coherent_abscont}
\end{align}

This paper will particularly focus on the coherent risk measures of the mean and the \ac{ES} \cite{rockafellar2002conditional}. The mean may be interpreted as having parameters $F=\textrm{id}_\R$ (identity) and $C = \{\zeta = \psi\}$.
The \ac{ES} at $\epsilon \in [0, 1)$ is defined as the expected value of $V$ conditioned that $V$ is above the $(1-\epsilon)$-level quantile \cite{rockafellar2002conditional}:
\begin{align}
    \mathrm{ES}_\epsilon(V) = \min \left\{ \lambda + \frac{1}{\epsilon} \mathbb{E}\left[(V-\lambda)_+\right] \; \middle| \; \lambda \in \mathbb{R} \right\}. \label{eq:cvar_param}
\end{align}
The \ac{ES}  has the expression in \eqref{eq:coherent_abscont} of (Equation 5.5 of \cite{follmer2010convex}):
\begin{subequations}
\label{eq:cvar_abscont}
\begin{align}
    F &= \textrm{id}_\R, & C = \left\{\frac{d\zeta}{d\psi} \leq \frac{1}{\epsilon}, \inp{1}{\zeta} = 1\right\}.
\end{align}
\end{subequations}

The \ac{ES} can also be represented as a domination-form \ac{LP} (Theorem 5.3 of \cite{miller2023chancepeak}):
\begin{subequations}
\label{eq:cvar_dom}
%\begin{align}
%    ES_\epsilon(\nu) = & \sup_{\psi, \hat{\psi} \in \Mp{\R}} \inp{\omega}{\psi} \\
%    & \epsilon\psi + \hat{\psi} = \nu \label{eq:cvar_dom_cond}\\
%    & \inp{1}{\psi} = 1. \label{eq:cvar_mass}
%\end{align}
\begin{align}
    \mathrm{ES}_\epsilon(V) = & \sup_{\nu, \hat{\nu} \in \Mp{\R}} \inp{\mathrm{id}_\R}{\nu} \\
    & \epsilon\nu + \hat{\nu} = \psi \label{eq:cvar_dom_cond}\\
    & \inp{1}{\nu} = 1. \label{eq:cvar_mass}
\end{align}
    \end{subequations}

Other instances of coherent risk measures include the Entropic Value-at-Risk \cite{ahmadi2012entropic}, the tail-conditional median \cite{kou2013external}, and the mean plus scaled standard deviation \cite{follmer2010convex}. Risk measures that are not coherent include the Value-at-Risk (quantile) and the mean plus scaled variance.

\subsection{Stochastic Processes and Occupation Measures}

Let $\mathbb{T}$ be a time domain, such as $\mathbb{T} = [0, T]$ in continuous time or $\mathbb{T} = \{k \Delta t\}_{k=0}^{k = T/\Delta t}$ in discrete-time. 
A stochastic process evolving in a state space $X \in \R^n$ may be represented as a time-indexed family of probability measures $\{\mu_t \in \Mp{X}\}_{t \in \mathbb{T}}$ \cite{oksendal2003stochastic}. This family of probability measures admits a time-shifting semigroup operator $\textrm{Shift}_{\tau}$ as $\textrm{Shift}_{\Delta t}(\mu_t) = \mu_{t+\Delta t}.$ 

In this work, we will restrict our discussion to stochastic processes that can uniquely be described by their \textit{generator} $\Lie_{\Delta t}$. Letting $\cs$ be the domain of the generator $\Lie_{\Delta t}$ and $v$ be a member of $\cs$, the generator of the stochastic process $\{\mu_t\}_{t \in \mathbb{T}}$ is defined, denoting $t' = t+\Delta t'$, as
\begin{equation}
    \Lie_{\Delta t} v(t, x)= \lim_{\Delta t' \rightarrow \Delta t} \frac{\inp{v(t', x)}{\mu_{t'}}- \inp{v(t,x)}{\mu_t}}{\Delta t'}. \label{eq:generator_lim}
\end{equation}

 For a discrete-time system with parameter $\lambda \in \Lambda$ distributed according to a probability distribution $\zeta(\lambda)$
\begin{align}
    x[t+\Delta t] &=  f(t, x[t], \lambda[t]), \qquad \lambda[t] \sim \zeta, \\
    \intertext{the associated generator is given by an expected value on $\lambda$}
    %\Lie_{\Delta t} v &= \left(\int_{\Lambda}v(t+\Delta t, f(t, x, \lambda)) d \zeta(\lambda) - v(t, x)\right)/\Delta t. \label{eq:generator_discrete}
    \Lie_{\Delta t} v(t,x) &= \frac{\mathbb{E}_\zeta[v(t+\Delta t, f(t,x,\lambda))] - v(t,x)}{\Delta t}. \label{eq:generator_discrete}
\end{align}
The domain of this discrete-time process $\Lie_{\Delta t}$ in \eqref{eq:generator_discrete} is $\cs = C([0, T] \times X)$.

The generator of a continuous-time stochastic process is defined as $\Lie = \lim_{\Delta t \rightarrow 0} \Lie_{\Delta t}$ (when this limit exists). One such continuous-time stochastic process is \iac{SDE}.
An It\^{o}-type \ac{SDE} defined by a drift $f$, diffusion term $g$, and Wiener process $dW$ is
\begin{equation}
\label{eq:sde}
    dx = f(t, x) dt + g(t, x) d{W}.
\end{equation}
The generator of the dynamics in \eqref{eq:sde} is
\begin{equation}
\label{eq:lie}
    \Lie v(t, x) = \partial_t v + f \cdot \nabla_x v + \frac{1}{2} g^T \left(\nabla^2_{xx}v\right)  g,
\end{equation}
in which the domain of $\Lie$ is $\cs = C^{1, 2}([0, T] \times X)$ (once-differentiable in $t$ and twice-differentiable in $x$).

% Scherer Psatz for Matrices \cite{scherer2006matrix}
\section{Time-Windowed Risk Setup}
\label{sec:risk_setup}

% \urg{The main theory}

This section presents preparatory material for the later relaxation of \eqref{eq:risk_window} into infinite-dimensional \acp{LP} in measures/continuous functions. It offers an illustrative example of instantaneous vs. time-windowed risk estimation, lists assumptions on the stochastic generator $\Lie$, and describes temporal support sets.

% This section will relax the risk program in \eqref{eq:risk_window} into infinite-dimensional \ac{LP}. 

% Primal-dual pairs of infinite-dimensional \ac{LP} in measures will be presented for mean and for \ac{ES} risk measures $\mathcal{R}$.

\subsection{Example of Risk Estimation}

To motivate instantaneous and time-windowed risk estimation, consider the (deterministic) black signal $p(t)$ plotted in Figure \ref{fig:demo_peak_inst_avg} up to a time horizon of $T=8$.

\old{
Figure \ref{fig:demo_peak_inst_avg} and Figure \ref{fig:demo_peak_inst_avg_multiple_s} give an intuition on the difference between the instantaneous peak and a time-averaged peak.  The red dotted lines plots the time-windowed average with $h=1.5$. The blue square at $t_i=2.0017$ is the instantaneous peak of $p_i =  1.5435$. The time-windowed average at $t_i$ is $\int_{t'=t_i-1}^{t_i} p(t') dt' =  0.0173$. The blue star at $t_w = 5.7502$ achieves the maximal time-windowed average peak of $p_w = 0.9379$, while the instantaneous value of $p(t_a)$ is   1.5438. The magenta curve of Fig. \ref{fig:demo_peak_inst_avg} is the region that is averaged to produce $p_w$.  }
\begin{figure}[!h]
    \centering
    \includegraphics[width=0.85\linewidth]{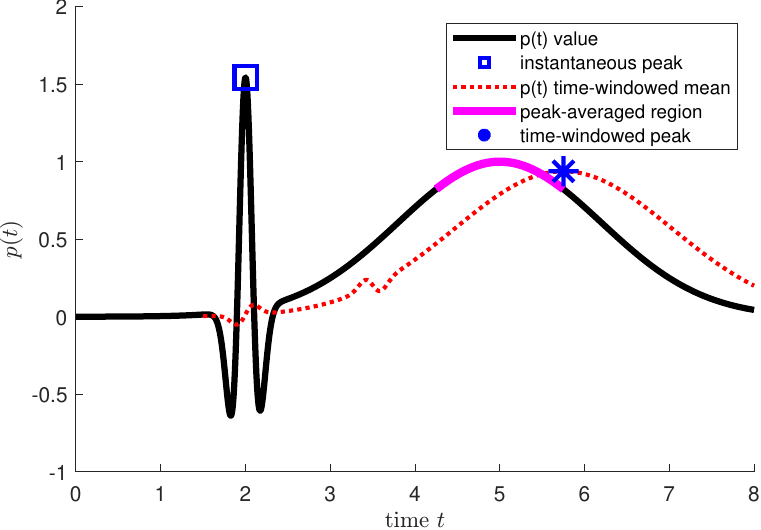}
    \caption{Comparison of instantaneous peak (square) and time-windowed average peak (star) of a signal $p(t)$ (black curve)}
    \label{fig:demo_peak_inst_avg}
\end{figure}

% The previous discussion solely focuses on the time-windowed mean. 
Letting $s$ be a stopping time, a measure $\mu_{p, s}$ can be chosen as the unique occupation measure supported on the graph of $(t, p(t))$ between times $t \in [s-h, s].$ The marginalized measure $(1/h)\pi_{\#}^p \mu_{p, s} $ is a probability distribution of the value of $p$, as weighted by time. 

The blue curve of Figure \ref{fig:signal_cdf} is the empirically evaluated Conditional Density Function constructed from the signal values $\mathcal{D}=\{p(t_0 + h(k/N))\}_{k=0}^{N}$  (with $t_0 = 1.25, \ h=1.5, \ N = 20000$). The red dotted line of Figure \ref{fig:signal_cdf} plots the mean value of $\mathcal{D}$ as  $p=0.0592$. The green dash-dotted line is the 90\% quantile value of 0.5387, and the yellow dotted line is 90\% \ac{ES} value of 1.1763. This example justifies the use of time-windowed risk measures for deterministic systems, which are a narrow specialization of more general stochastic systems.
\begin{figure}[!h]
    \centering
    \includegraphics[width=\linewidth]{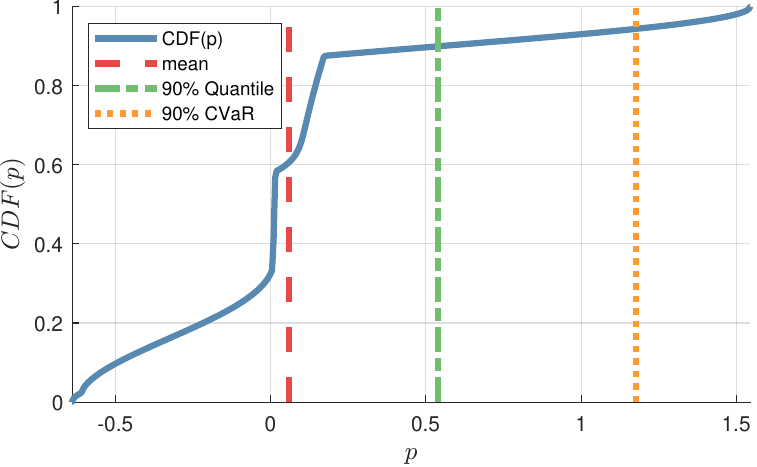}
    \caption{Empirical Cumulative Density Function of the value of $p(t)$ in times $[1.25, 2.75]$}
    \label{fig:signal_cdf}
\end{figure}

\begin{figure*}
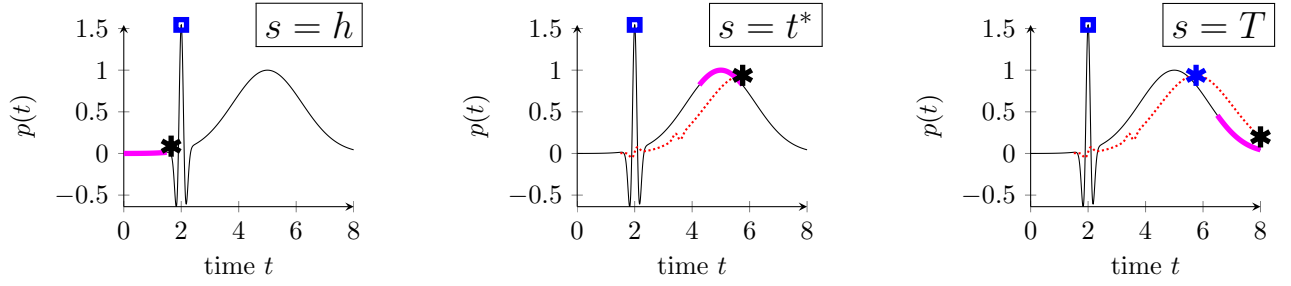

    \centering
    \begin{tikzpicture}
        \input{img/plot_data_first_window}
    
        \path (axis1.south) ++ (4.5cm,0) coordinate (test);
    
        \input{img/plot_data_second_window}
    
        \path (axis2.south) ++ (4.5cm,0) coordinate (test);
    
        \input{img/plot_data_third_window}
        
        % \draw[-Stealth] (axis1.south)++ (0,-1.3cm) -- ++(axis3.south) node[below] {$s$};
        % \draw (axis1.south)++ (-0cm,-1.3cm) -- ++(-0cm,-.2cm) node[below] {$h$};
        % \draw (axis2.south)++ (-0cm,-1.3cm) -- ++(-0cm,-.2cm) node[below] {$t_1$};
        % \draw (axis3.south)++ (-0cm,-1.3cm) -- ++(-0cm,-.2cm) node[below] {$T$};
        % \node[right = 0.4cm of axis1.east]  {\Large \textbf{$\dots$}};    
        % \node[right = 0.4cm of axis2.east]  {\Large \textbf{$\dots$}};    
        \node[right = 0.1cm of axis1.north]  {\fcolorbox{black}{white}{\Large \textbf{$s=h$}}};  
        \node[right = 0.1cm of axis2.north]  {\fcolorbox{black}{white}{\Large \textbf{$s=t^*$}}}; 
        \node[right = 0.1cm of axis3.north]  {\fcolorbox{black}{white}{\Large \textbf{$s=T$}}};   
    \end{tikzpicture}
    \caption{\old{Comparison of instantaneous peak (square) to time-windowed average peak (blue star) of a signal $p(t)$ from Figure \ref{fig:demo_peak_inst_avg} with stopping time $s \in \{h, \ t^*, \ T\}$. The magenta region highlights $p(t)$ in times $[s-h, s]$. The time-windowed average in this range is marked with the black star. The red dotted line depicts all time-windowed averages up to $t\leq s$.}}
    \label{fig:demo_peak_inst_avg_multiple_s}
\end{figure*}

\subsection{Assumptions}

The following assumptions will be posed throughout this paper:

\begin{itemize}
    \item[A1] The sets 
    $X_0, X$, and $[0, T]$ are all compact.
    \item[A2] The function $p(x)$ is continuous.
    \item[A3] Trajectories stop according to the rule $\min(T, \tau_X)$.
    \item[A4] The set of functions $\cs = \textrm{dom}(\Lie)$ satisfies $\cs \subseteq C([t_0, T] \times X)$
    % $\textrm{Hom}(C([t_0, T] \times X), C([t_0, T] \times X))$ 
    % {\color{Purple}Matt: what is \textrm{Hom}?} 
    with $1 \in \cs$ and $\Lie 1 = 0$.
    \item[A5] The set $\cs$ is closed under multiplications and separates points (for all distinct $(t_1, x_1), (t_2, x_2) \in [0, T] \times X$, there exists a $v \in \cs$ with $v(t_1, x_1) \neq v(t_2, x_2)$).    
    \item[A6] A countable set $\{v_k\} \in \cs$ exists such that $\forall v \in \cs,$ the tuple $(v, \Lie v)$ is contained within the bounded pointwise closure of $\textrm{span}(\{(v_k, \Lie v_k)\})$. 
    % {\color{Purple}Matt: what is $A$ here?}
\end{itemize}

Assumptions A4-A6 are technical assumptions on the stochastic generator derived from \cite{cho2002linear}. Acceptable generators include many \acp{SDE}, L\'{e}vy processes, and discrete-time Markov processes.

\subsection{Temporal Supports}

The \ac{LP} formulation of Problem \ref{prob:risk_window} that we propose in this paper requires the addition of an augmented time variable $s$ to the stochastic dynamics. This variable $s$ will be constant, with a dynamical law of $\dot{s}=0$ for continuous time or $s_{t+\Delta t} = s_t$ for discrete-time. The value of $s$ will represent the stopping time of the stochastic process. Given a generator $\Lie$ from dynamics in \eqref{eq:risk_window_dynamics} defined in terms of $(t, x)$, the augmented dynamics $\hat{\Lie}$ in terms of $(s, t, x)$ is 
\begin{align}
    \hat{\Lie}: \quad  \textrm{$s(t)$ is constant}, \quad  x(t) \textrm{\ follows  $\Lie$} \label{eq:lie_hat}
\end{align}

Continuous-time temporal support sets in $(s, t)$ can be defined as
% \old{
% In order to phrase linear programs for \eqref{eq:risk_window}, we will first define a new constant state $s$ to dynamics (with $\dot{s} = 0$). The generator $\hat{\Lie}$ will refer to the augmented dynamics in $(s, x)$:
% \begin{align}
%     \hat{\Lie}: \quad  \dot{s} = 0, \quad  x(t) \textrm{\ follows \ } \Lie  \label{eq:lie_hat}
% \end{align}
% The constant state $s$ will serve as the stopping time $t^*$ in optimization problem \eqref{eq:risk_window}. In the case of a discrete-time system, the constant $s$ dynamics in \eqref{eq:lie_hat} will satisfy the rule $s_{t+\Delta t} = s_t$. We thereby define the augmented continuous-time temporal-support sets as}
\begin{subequations}
\label{eq:omega_support_c}
\begin{align}    
    \Omega_+^c &= \{(s, t) \in [h, T] \times [0, T] \mid t \in [s-h, s] \} \\
    \Omega_-^c &= \{(s, t) \in [h, T] \times [0, T] \mid t \in [0, s-h] \},
\end{align}
\end{subequations}
and the discrete-time temporal-support sets as
\begin{subequations}
\label{eq:omega_support_d}
\begin{align}    
    \Omega_+^d &= \{(s, t) \in [h, T] \times [0, T] \mid t \in [s-h, s] \} \\
    \Omega_-^d &= \{(s, t) \in [h, T] \times [0, T] \mid t \in [0, s-h-\Delta t] \}.
\end{align}
\end{subequations}

For ease of notation, the temporal support sets $\Omega_\pm$ will refer to $\Omega_\pm^c$ in continuous-time or alternatively $\Omega_\pm^d$ in discrete-time.

% The temporal support sets $\Omega_\pm$ will refer to:
% \begin{align}
% \label{eq:omega_support}
%     \textrm{Continuous-Time:} & \qquad \Omega_\pm = \Omega_\pm^c \\
%     \textrm{Discrete-Time:} & \qquad\Omega_\pm = \Omega_\pm^d \nonumber
% \end{align}

We will also introduce a temporal augmentation map $\varphi$:
\begin{equation}
    \varphi: (s, x) \mapsto (s, s, x).
\end{equation}

\section{Time-Windowed Mean Programs}
\label{sec:mean_lp}
This section formulates infinite-dimensional \acp{LP} for the time-windowed risk analysis Problem \ref{prob:risk_window} when $\mathcal{R}$ is chosen to be the mean. A value of $\epsilon \in [0, 1)$ will assumed to be fixed throughout this section.

\subsection{Measure Program}
\old{
The following measures will be defined to formulate the mean-type time-windowed risk \ac{LP} from \eqref{eq:risk_window}:
\begin{subequations}
\label{eq:risk_meas_def}
\begin{align}
    \mu_0(s, x) &\in \Mp{[h, T]\times  X_0}  & & \textrm{Initial}  \\ 
    \mu_\tau(s, x) &\in \Mp{[h, T] \times X}  & & \textrm{Terminal} \label{eq:risk_meas_def_term} \\
    \mu_+(s, t, x) &\in \Mp{\Omega_+ \times X}  & & \textrm{Risk Occ.} \label{eq:risk_meas_def_occ} \\
    \mu_-(s, t, x) & \in \Mp{\Omega_- \times X} & & \textrm{Past Occ.} \label{eq:risk_meas_def_past}
\end{align}
\end{subequations}

% The notation $[\Lie; 0]^\dagger$ will refer to the following relation holding for all test functions $v \in \cs$

The mean-type measure program for \eqref{eq:risk_window} is:
\begin{prob}
\label{prob:risk_mean_meas}
Find an initial measure $\mu_0,$ a terminal measure $\mu_\tau$, a risk occupation measure $\mu_+$, and a past occupation measure $\mu_-$ to supremize:
\begin{subequations}
\label{eq:risk_mean_meas}
\begin{align}
p^* = & \ \sup \quad \inp{p}{\mu_+}/h \label{eq:risk_mean_meas_obj} \\
  \textrm{s.t.} \   & \varphi_\# \mu_\tau = \delta_0 \otimes\mu_0 +  \hat{\Lie}^\dagger (\mu_- + \mu_+) \label{eq:risk_mean_meas_flow}\\
    & \inp{1}{\mu_0} = 1 \label{eq:risk_mean_meas_prob}\\
    & \inp{1}{\mu_+} = h \label{eq:risk_mean_meas_window} \\
    & \textrm{Support constraints in \eqref{eq:risk_meas_def}} \label{eq:risk_mean_meas_def}
\end{align}
\end{subequations}
\end{prob}

The sum $\mu_- + \mu_+$ serves as the relaxed occupation measure for dynamics following $\hat{\Lie}$ in \eqref{eq:lie_hat}. The mean value of $p$ is only evaluated in the range $t \in [s-h, s],$ which is enforced by the $\Omega_+$ support constraint in \eqref{eq:risk_meas_def_term}. Constraint \eqref{eq:risk_mean_meas_window} ensures that the `clock' for $\mu_+$ will elapse at exactly $h$ time units.
The notation $\varphi_\# \mu_\tau$ from \eqref{eq:risk_mean_meas_flow}  can be read as $\forall v \in C([h, T]^2 \times X): \ \inp{v(s, t, x)}{\varphi_\# \mu_\tau} = \inp{v(s, s, x)}{\mu_\tau}.$

\begin{thm}
\label{thm:risk_mean_upper}
    Under Assumption A3, program \eqref{eq:risk_mean_meas} is an upper-bound on \eqref{eq:risk_window} ($p^* \geq P^*$) when $\mathcal{R}$ is the mean. 
\end{thm}
\begin{proof}
This proof will construct feasible measures in \eqref{eq:risk_meas_def} from a feasible point $(t^*, x_0^*) \in [h, T] \times X_0$ of \eqref{eq:risk_window}.  The initial measure $\mu_0$ may be set to $\mu_0 = \delta_{s=t^*, \  x = x_0}$. Letting $x(t^*)$ represent the probability distribution of the process at time $t^*$, the terminal measure may be chosen as $\mu_\tau = \delta_{t=t^*, \ s=t^*} \otimes x(t^*)$. Defining $\mu(t, x)$ as the occupation measure of the process $x(t)$ between times $0$ and $t^*$, let $\mu_{[0, t^*-h]}(t, x)$ and $\mu_{[t^*-h, t]}(t, x)$  refer to the restrictions of $\mu$ in times $[0, t^*-h]$ and $[t^*-h, t^*]$ respectively. The decomposed relaxed occupation measures may therefore be chosen as $\mu_- = \delta_{s=t^*} \otimes \mu_{[0, t^*-h]}$ and $\mu_+ = \delta_{s=t^*} \otimes \mu_{[t^*-h, t^*]}$. Note that constraint \eqref{eq:risk_mean_meas_window} is satisfied because $\mu_+$ is defined over exactly $h$ time units.

The objective term in \eqref{eq:risk_mean_meas_obj} implements the time-averaging operation in \eqref{eq:risk_window} w.r.t. choosing $\mathcal{R}$ as the mean.
The upper-bound relation $p^* \geq P^*$ therefore holds, because this construction procedure induces an injective map from the feasible points $(t^*, x_0^*)$ to feasible points of \eqref{eq:risk_mean_meas_flow}-\eqref{eq:risk_mean_meas_def}
\end{proof}

\begin{thm}
\label{thm:no_relaxation_mean}
    Under Assumptions A1-A6, there is no relaxation gap ($p^* = P^*$) when $\mathcal{R}$ is the mean.
\end{thm}
\begin{proof}
Let $(\mu_0, \mu_\tau, \mu_J)$ (with $\mu_J \in \Mp{[h, T] \times [0, T] \times X}$) satisfy
\begin{equation}
    \mu_\tau = \delta_0 \otimes\mu_0 +  \hat{\Lie}^\dagger \mu_J \label{eq:risk_mean_meas_flow_J}
\end{equation}
    By Assumptions A1-A6 and Theorem 3.3 of \cite{cho2002linear}, every feasible tuple $(\mu_0, \mu_\tau, \mu_J)$ satisfying \eqref{eq:risk_mean_meas_flow}-\eqref{eq:risk_mean_meas_def} is supported on the graph of stochastic trajectories of $\hat{\Lie}$.    
    
    Furthermore, the marginalization $\pi^t_{\#} \mu_J$ is absolutely continuous to the Lebesgue distribution defined over $[0, T]$ (continuous-time), or to a spike train $\sum_{k=0}^{T/\Delta t} \delta_{t=k \Delta t}$ (discrete-time). As such, there exists a decomposition of $\mu_J$ into  $\mu_-$ and $\mu_+$ through the $(s, t)$ temporal supports in \eqref{eq:omega_support_c} or \eqref{eq:omega_support_d} subject to the requirement \eqref{eq:risk_mean_meas_window}. Such a decomposition may be achieved through the restriction procedure from Theorem 3.1 of \cite{henrion2009approximate}, in which $\mu_+$ is chosen to maximize $\inp{t}{\mu_+}$ such that $\mu_+ + \mu_- = \mu_J, \ \inp{1}{\mu_+}=h$ under \eqref{eq:risk_meas_def_occ}-\eqref{eq:risk_meas_def_past}.

    It therefore holds that the tuple $(\mu_0, \mu_\tau, \mu_- + \mu_+)$ is supported on the graph of a stochastic process, and all measures satisfy the support constraints in \eqref{eq:risk_meas_def}.  Given that $\mu_0$ is a probability distribution by \eqref{eq:risk_mean_meas_prob}, the term $\inp{p}{\mu_+}/\tau$ in the objective \eqref{eq:risk_mean_meas_obj} evaluates to the integral in \eqref{eq:risk_window_obj}. No relaxation gap is therefore proven under A1-A6.
\end{proof}
}

% Measure solutions of Problem \ref{prob:risk_mean_meas} satisfy a boundedness property when assumptions A1-A4 are imposed:

\subsection{Function Program}

% \urg{Derive the dual. Prove strong duality. Keep the machine pumping.}

A dual \ac{LP} for \eqref{eq:risk_mean_meas_obj} will be expressed through the declaration of  variables $v \in \cs([h, T] \times [0, T] \times X)$ and $\gamma,\xi \in \R$. 

\begin{prob}
\label{prob:risk_mean_cont}
Find an auxiliary function $v$ and scalars $(\gamma, \xi)$ to infimize:
% This \ac{LP} will have the form:
\begin{subequations}
\label{eq:risk_mean_cont}
\begin{align}
    d^* 
    = & \inf_{v, \ \gamma, \ \xi} \quad \gamma + h \xi & \\
   \textrm{s.t.} \qquad   & \forall x \in X_0: \nonumber \\
   & \qquad \gamma \geq  v(s,0, x)   \label{eq:risk_mean_cont_init}\\
    & \forall (s, x) \in [h, T] \times X:  \nonumber \\
    &\qquad  v(s, s, x) \geq 0 & & \label{eq:risk_mean_cont_term}\\
    &\forall (s, t, x) \in \Omega_+ \times X: \nonumber \\
    & \qquad \hat{\Lie}v(s, t, x) +p(x)/h\leq  \xi& \label{eq:risk_mean_cont_occ_fw}  \\
    & \forall (s, t, x) \in \Omega_- \times X: \nonumber \\
    & \qquad \hat{\Lie}v(s, t, x) \leq  0  & & \label{eq:risk_mean_cont_occ_bw} \\
    &v \in \cs([h, T] \times [0, T]\times X) & \label{eq:risk_mean_cont_poly} \\
    & \gamma, \ \xi \in \R.
\end{align}
\end{subequations} 
\end{prob}

\begin{thm}
\label{thm:duality_mean}
    Strong duality between problems \eqref{eq:risk_mean_meas} and \eqref{eq:risk_mean_cont} ($p^* = d^*$) will hold with attainment of optimality under assumptions A1-A6.
\end{thm}
\begin{proof}
    See Appendix \ref{app:duality_mean}.
\end{proof}

\begin{rmk}
    The strong duality statement from Theorem \ref{thm:duality_mean} allows a control practitioner to get the same answer regardless of whether they implement the measure side (Problem \ref{prob:risk_mean_meas}) or the continuous-function side (Problem \ref{prob:risk_mean_cont}) of the time-windowed risk analysis Problem \ref{prob:risk_window} during discretization.
\end{rmk}

\section{Time-Windowed CVaR Programs}
\label{sec:cvar_lp}
This section defines \acp{LP} for Program \ref{prob:risk_window} when the \ac{ES} is used as the risk measure $\mathcal{R}$. Modifications for the \ac{ES} framework (as compared to the mean-type measure program detailed in Section \ref{sec:mean_lp}) will be specifically noted. 
% We will assume that 

% \urg{Extend our Chance-Peak CVAR work. CVAR measure and function programs.}

\subsection{Measure Program}

The time-windowed \ac{ES} estimation program will use the domination-form expression for the \ac{ES} from Equation \eqref{eq:cvar_dom}. Given compactness of the state space $X$ and continuity of the objective function $p$ by Assumptions A1 and A2 respectively, it results that the following two values are finite and are attained:
\begin{align}
    p_{\min} &= \min_{x \in X} p(x), & p_{\max} &= \max_{x \in X} p(x).
\end{align}

It therefore holds that the set $[0, T] \times X \times [p_{\min}, p_{\max}]$ is compact under Assumptions A1 and A2. We introduce the following measures to supplement the measures in \eqref{eq:risk_meas_def}:
\begin{subequations}
\label{eq:risk_meas_def_cvar}
\begin{align}
    \nu &\in \Mp{[p_{\min}, p_{\max}]}  & & \textrm{\ac{ES} Evaluation}  \\ 
    \hat{\nu} &\in \Mp{[p_{\min}, p_{\max}]}  & & \textrm{\ac{ES} Slack} \label{eq:risk_meas_def_cvar_slack_term}.
\end{align}
\end{subequations}

The \ac{ES}-type measure \ac{LP} for \eqref{eq:risk_window} is:
\begin{prob}
\label{prob:risk_cvar_meas}
Find an initial measure $\mu_0,$ a terminal measure $\mu_\tau$, a risk occupation measure $\mu_+$, a past occupation measure $\mu_-$, and \ac{ES} decomposition measures $\nu, \hat{\nu}$ to supremize:
\begin{subequations}
\label{eq:risk_cvar_meas}
\begin{align}
p^*_{ES} = & \ \sup \quad \inp{\textrm{id}_\R}{\nu} \label{eq:risk_cvar_meas_obj} \\
  \textrm{s.t.} \   & \mu_\tau = \delta_0 \otimes\mu_0 +  \hat{\Lie}^\dagger (\mu_- + \mu_+) \label{eq:risk_cvar_meas_flow}\\
    & \inp{1}{\mu_0} = 1 \label{eq:risk_cvar_meas_prob}\\
    & \inp{1}{\mu_+} = h \label{eq:risk_cvar_meas_window} \\
    & \inp{1}{\nu} = 1 \label{eq:risk_cvar_meas_cvarmass}\\
    & \epsilon\nu + \hat{\nu} = p_\#\pi^x_\#(\mu_+/h) \label{eq:risk_cvar_meas_cvarcon}\\
    & \textrm{Support constraints in \eqref{eq:risk_meas_def} and \eqref{eq:risk_meas_def_cvar}} \label{eq:risk_cvar_meas_def}
\end{align}
\end{subequations}
\end{prob}
% \inp{p}{\mu_+}/h

Constraints \eqref{eq:risk_cvar_meas_cvarmass}-\eqref{eq:risk_cvar_meas_cvarcon} are added to the formulation in Problem \ref{prob:risk_mean_meas}. Note that the objective in \eqref{eq:risk_cvar_meas_obj} (mean of $\nu$) is also different than the objective in \eqref{eq:risk_mean_meas_obj} ($\inp{p}{\mu_+}/h$). 

\begin{rmk}
    The term $p_\# (\pi^x_\# \mu_+ /h)$ can be understood w.r.t. its pairing with  functions $\forall g \in C([p_{\min}, p_{\max}])$ as
    \begin{align}
        \inp{g}{p_\# (\pi^x_\# \mu_+ /h)} = \inp{g(p(x))}{\mu_+(s, t, x)}/h.
    \end{align}
\end{rmk}

\begin{rmk}
\label{rmk:cvar}
The objective in \eqref{eq:risk_cvar_meas_obj} together with constraints \eqref{eq:risk_cvar_meas_cvarcon}-\eqref{eq:risk_cvar_meas_cvarmass} implements the \ac{ES} expression in \eqref{eq:cvar_dom} with respect to the univariate probability distribution $p_\# (\pi^x_\# \mu_+ /h)$.
% \end{prop}
\end{rmk}

    The following \ac{ES}-type bounding and no-relaxation-gap theorems will be proven by extending arguments from the mean-type  Theorems \ref{thm:risk_mean_upper} and \ref{thm:no_relaxation_mean}.
\begin{thm}
\label{thm:risk_cvar_upper}
    When $\mathcal{R}$ is the \ac{ES} and Assumption A3 holds, Problem \ref{prob:risk_cvar_meas} will have an optimal value satisfying $p^*_{ES} \geq P^*$ from Problem \ref{prob:risk_window}.     
    \end{thm}
    \begin{proof}
        Let $(t^*, x_0^*) \in [h, T] \times X_0$ be a feasible (but not necessarily optimal) point of Problem \ref{prob:risk_window}. Measures $(\mu_0, \mu_\tau, \mu_+, \mu_-)$ can be constructed from $(t^*, x_0^*)$ to satisfy constraints in \eqref{eq:risk_cvar_meas_flow}-\eqref{eq:risk_cvar_meas_window} subject to the support constraints in \eqref{eq:risk_meas_def}. 
        We now consider the construction of feasible measures $\nu, \hat{\nu}$ from the point $(t^*, x_0^*).$ 
        Define $\xi^* = p_\# (\pi^x_{\#} \mu_{[t^*-h, t^*]}(t, x)/h)$ 
        as the probability distribution formed by the evaluation of $p$ along the graph of $x(t)$ between $[t^*-h, t^*]$. 
        Noting that  univariate probability distribution $\xi^*$ satisfies the support constraint $\xi^* \in \Mp{[p_{\min}, p_{\max}]}$, the distribution can be split into $\nu^*, \hat{\nu^*}$ such that the top $\epsilon$ mass of $\xi^*$ is stored in $(1/\epsilon)\nu^*$ while the lower $(1-\epsilon)$ mass of $\xi^*$ is left to the slack $\hat{\nu}^*$ (expression in \eqref{eq:cvar_dom}). The expectation of such a $\nu^*$ is equal to the \ac{ES} of $p$ in the time-window $[t^*-h, t^*]$. The theorem is proven, given that measures from \eqref{eq:risk_meas_def} and \eqref{eq:risk_cvar_meas_def} can be constructed from any point $(t^*, x_0^*).$
    \end{proof}
    
    \begin{thm}
    \label{thm:no_relaxation_cvar}
    When Assumptions A1-A6 are imposed for the \ac{ES} problem, Problem \ref{prob:risk_cvar_meas} will have no relaxation gap to Problem \ref{prob:risk_window} $(p^*_{ES} = P^*)$.
\end{thm}
\begin{proof}
    Following from the proof of Theorem \ref{thm:no_relaxation_mean}, under Assumptions A1-A6, the measures $(\mu_0, \mu_\tau, \mu_{+}, \mu_{-})$ will be properly supported on the graph of a stochastic process. Because $(\nu, \hat{\nu})$ from \eqref{eq:risk_cvar_meas_def} are the  measures associated with the \ac{ES} of $p_\# (\pi^x_\# \mu_+/h)$ from    
    relation \ref{eq:cvar_dom}, feasible points of Problem \ref{prob:risk_cvar_meas} remain supported on the graph of stochastic processes. No-relaxation-gap is therefore proven to the \ac{ES} objective.
\end{proof}

\begin{prop}
    The \ac{ES} program \eqref{eq:risk_cvar_meas} and mean program \eqref{eq:risk_mean_meas} are related by $p^*_{ES} \geq p^*$ (even when no assumptions from A1-A6 are imposed).
\end{prop}
\begin{proof}
    The mean is an instance of \ac{ES} with parameter $\epsilon=1$. Enforcing $\epsilon=1$ in constraint \eqref{eq:risk_cvar_meas_cvarcon} ensures that $\hat{\nu} = 0$ and $\nu = p_\# \pi^x_\#(\mu_+/h)$, in which case $\inp{\textrm{id}_\R}{\nu} = \inp{p}{\mu_+}/h$. In the $\epsilon=1$ case, the objectives values of Problems \ref{prob:risk_cvar_meas} and \ref{prob:risk_mean_meas} are equal.

    When $\epsilon<1$, the feasible $\nu$ measure is allowed to exclude lower-$p$-valued mass from $p_\# \pi^x_\#(\mu_+/h)$. In the $\epsilon <1$ case, the objectives are therefore related as $p^*_{ES} \geq p^*$.
\end{proof}

    % By the definition of the \ac{ES} as the mean above the $\epsilon$-quantile,  for any random variable $V$ and parameters $\epsilon' < \epsilon$ it holds that $ES_{\epsilon}(V) \leq ES_{\epsilon'}(V)$. 

\subsection{Function Program}

The dual \ac{LP} for \eqref{eq:risk_cvar_meas} will have a scalar variable $\beta$ and a univariate function variable $w$ as well as the $(v, \gamma, \xi)$ from \eqref{eq:risk_mean_meas}. The statement of the continuous-function \ac{LP} \ac{ES}-type time-windowed risk analysis:
\begin{prob}
\label{prob:risk_cvar_cont}
Find an auxiliary function $v$, a univariate function $w$, and scalars $(\gamma, \xi, \beta)$ to infimize:
% This \ac{LP} will have the form:
\begin{subequations}
\label{eq:risk_cvar_cont}
\begin{align}
    d^*_{ES} 
    = & \inf_{v, \ \gamma, \ \xi} \quad \gamma + h \xi +\beta & \\
   \textrm{s.t.} \qquad   & \forall x \in X_0: \nonumber \\
   & \qquad \gamma \geq  v(s,0, x)   \label{eq:risk_cvar_cont_init}\\
    & \forall (s, x) \in [h, T] \times X:  \nonumber \\
    &\qquad  v(s, s, x) \geq 0 & & \label{eq:risk_cvar_cont_term}\\
    &\forall (s, t, x) \in \Omega_+ \times X: \nonumber \\
    & \qquad \hat{\Lie}v(s, t, x) +w(p(x))/h\leq  \xi& \label{eq:risk_cvar_cont_occ_fw}  \\
    & \forall (s, t, x) \in \Omega_- \times X: \nonumber \\
    & \qquad \hat{\Lie}v(s, t, x) \leq  0  & & \label{eq:risk_cvar_cont_occ_bw} \\
    & \forall q \in [p_{\min}, p_{\max}] \nonumber \\
    & \qquad \epsilon w(q) + \beta \geq 0 \label{eq:risk_cvar_cont_occ_cvar}\\
    &v \in \cs([h, T] \times [0, T]\times X) & \\
    &w \in C_+([p_{\min}, p_{\max}]) 
    \label{eq:risk_cvar_cont_poly} \\
    & \gamma, \ \xi, \ \beta \in \R.
\end{align}
\end{subequations} 
\end{prob}

\begin{thm}
\label{thm:duality_cvar}
Under assumptions A1-A6, strong duality ($p^*_{ES} = d^*_{ES}$) will hold between \eqref{eq:risk_cvar_meas} and \eqref{eq:risk_cvar_cont}.
    % Strong duality between problems \eqref{eq:risk_cvar_meas} and \eqref{eq:risk_cvar_cont} ($p^* = d^*$) will hold with attainment of optimality under assumptions A1-A6.
\end{thm}
\begin{proof}
    See Appendix \ref{app:duality_cvar}.
\end{proof}
\section{Time-Windowed Average Finite Truncation}
\label{sec:risk_lmi}
\old{
The infinite-dimensional \acp{LP} presented in Sections \ref{sec:mean_lp} and \ref{sec:cvar_lp} must be truncated into finite-dimensional programs in order to be computationally tractable. This section will first present the moment-\ac{SOS} hierarchy of \acp{SDP} \cite{lasserre2009moments}, and will then apply this hierarchy towards grid-free truncation of Problems \ref{prob:risk_mean_meas} and \ref{prob:risk_cvar_meas}.

\subsection{Moment-SOS Background}
%\urg{Fill in \ac{SOS} background for proofs of polynomial nonnegativity.the works.}

Problem~\eqref{eq:risk_mean_meas}  involves infinite-dimensional decision variables (namely Radon measures) as well as infinite-dimensional linear constraints of the form $\mathcal{A} \, \mu = b$ such as the Kolmogorov equation~\eqref{eq:risk_mean_meas_flow}. To cope with such infinite dimension, the moment-SOS framework consists of three steps: \textit{(i)} testing infinite dimensional constraints against polynomials, \textit{(ii)} representing Borel measures $\mu \in \Mp{X}$ with their moments $y_\alpha = \inp{x^\alpha}{\mu}$ and \textit{(iii)} truncating the degree of both moment variables and constraints to a finite value $k < \infty$. Step \textit{(i)} is a direct application of the Stone-Weierstra\ss{} theorem, while step \textit{(ii)} is made possible by Putinar's Positivstellensatz~\cite[Lemma 3]{putinar1993compact}, which characterizes moment sequences among multi-indexed real sequences:

\begin{lem} \label{lem:Psatz}
    Let $g_1,\ldots,g_m \in \R[x]$ define a \ac{BSA} set
    $$ X = \{x \in \R^n \; | \; g_1(x) \geq 0 , \ldots, g_m(x) \geq 0\}. $$
    Let $y = (y_\alpha)_{\alpha \in \N^n} \in \R^{\N^n}$ be a multi-indexed real sequence with associated Riesz linear functional
    $$ L_y = \left\{\begin{array}{ccc}
         \R[x] & \longrightarrow & \R \\
         \sum_{\alpha} p_\alpha \, x^\alpha & \longmapsto & \sum_\alpha p_\alpha \, y_\alpha. 
    \end{array}\right. $$
    Then, assuming the Archimedean property, i.e. that $g_m(x) = R^2 - x^\top x$ (which is equivalent to compactness of $X$, up to adding a redundant ball constraint to its description), there exists a Radon measure $\mu\in\Mp{X}$ such that
    $$ \forall \alpha \in \N^n, \quad y_\alpha = \inp{ x^\alpha }{ \mu } $$
    if and only if the following quadratic forms are nonnegative:
    \begin{subequations} \label{eq:PSatz}
        \begin{align}
            Q_y & = p \longmapsto L_y(p^2) \geq 0 , \label{eq:PSatz_moment}\\ Q_{g_iy} & = p\longmapsto L_y(g_i \, p^2) \geq 0. \label{eq:PSatz_local} \\ \notag
        \end{align}
    \end{subequations}
\end{lem}

Notice that when the conditions of Lemma~\ref{lem:Psatz} are met, it holds that
$$ \forall p \in \R[x], \quad L_y(p) = \inp{p}{\mu}.$$
Hence, Lemma~\ref{lem:Psatz} means that it is possible to replace Radon measures in~\eqref{eq:risk_mean_meas} with real sequences $y$ complemented with appropriate positivity constraints on $(Q_y, Q_{g,y})$.

Step \textit{(iii)} truncates infinite dimensional moment sequences to finite size moment vectors. For a degree $k \in \N$ and a polynomial $g \in \R[x]$, let $M_k(g y)$ be the symmetric square matrix of size $\binom{n+k-\ceil{\deg g/2}}{n}$ defined by
\begin{equation}
    M_k(g y)_{\alpha, \beta} = L_y(g(x) \, x^{\alpha+\beta}) = \textstyle\sum_{\gamma} g_{\gamma} y_{\alpha+ \beta + \gamma}.
\end{equation}
The matrix $M_k(g y)$ is equal to the top corner of the inifinite-dimensional matrix representing the quadratic form $Q_{g  y}$ in the basis of monomials. The degree-$k$ truncation of a measure \ac{LP} in step (iii) involves replacing each measure variable by a sequence of (pseudo)-moments $y$ up to degree $2k$, imposing that $M_k(y)$ and $M_k(g_i y)$ are each \ac{PSD} matrices, and replacing linear constraints in measures $\mathcal{A} \mu = b$ with a finite number of constraints in the pseudo-moments (e.g. $L_y(\inp{x^\alpha}{b-\mathcal A \, \mu})=0$, $\forall |\alpha| \leq k$). Each degree-$k$ truncation may be solved by \ac{SDP} algorithms, and the process of increasing $k \rightarrow \infty$ to achieve better bounds is the moment-\ac{SOS} hierarchy. 
% which is proved to vanish when the degree bound $k$ tends to infinity~\cite{tacchi2022convergence}.
% In the context of a measure \ac{LP} with a supremization objective, the 
% The moment-\ac{SOS} is the process of converting a measure \ac{LP} into 
% constraints $\mathcal{A} \mu = b$ to a finite number of constraints $\inp{x^\alpha}{b-\mathcal A \, \mu}=0$, $|\alpha| \leq k$, and positive semidefiniteness constraints~\eqref{eq:PSatz} to finite size \ac{LMI}, hence resulting in finite dimensional \ac{SDP} problems that can be numerically solved on a regular computer. This step \textit{(iii)} is called the moment hierarchy, and introduces a relaxation gap, which is proved to vanish when the degree bound $k$ tends to infinity~\cite{tacchi2022convergence}.
}
\subsection{Peak Time-Windowed Mean Moment-SOS Programs}

In order to apply the Moment-\ac{SOS} hierarchy, assumptions of polynomial structure must be imposed:
\begin{itemize}
    \item[A7] The function $p(x)$ is polynomial.
    \item[A8] The generator $\Lie$ sends polynomials to polynomials ($\forall z \in \R[t, x], \ \Lie \in \R[t, x]$).
    \item[A9] The states sets $X_0, X$ have an Archimedean \ac{BSA} representation:
    \begin{align*}
        X & = \{x \in \R^n \; | \; g_1(x) \geq 0, \ldots, g_m(x) \geq 0 \} \\
        X_0 & = \{x \in \R^n \; | \; g_{01}(x) \geq 0 , \ldots, g_{0\ell}(x) \geq 0\}.
    \end{align*}
\end{itemize}

\old{
To each degree $k \in \N$  and generator $\Lie$, the associated  dynamics degree $\tilde{k} \geq k$ can be defined as %\urg{make this degree explicit}
\begin{align}
    \tilde{k} = \left\lceil\max_{v \in \R_{2k}[t, x]} \frac{\deg(\Lie v)}{2}\right\rceil.
\end{align}
In addition, for $y = (y_\alpha)_{|\alpha| \leq 2k} \in \N^n_{2k}$ and $g \in \R_{2k}[x]$, let $\hat{k} = k - \ceil{\nicefrac{\deg(g)}{2}}$ and define the localizing matrix $M_k(g\, y)$ as the matrix representation of $Q_{g y}$ in a basis of $\R_{\hat{k}}[x]$.

We define the following polynomials describing $[h,T]$ and $\Omega_\pm$:
\begin{align*}
    g_h(s) & = (T-s)(s-h) \\
    %g_\tau(s,t) & = s-t \\
    g_+(s,t) & = (s-t)(t-s+h) \\
    g_-(s,t) & = \begin{cases}
        (s-t-h)t & \textrm{Continuous-Time} \\
        (s-t-h-\Delta t) t & \textrm{Discrete-Time}
    \end{cases}
\end{align*}
so that it holds
\begin{align*}
    [h,T] & = \{s \in \R \; | \; g_h(s) \geq 0\} \\
    \Omega_+ & = \{(s,t) \in \R^2 \; | \; g_h(s) \geq 0, g_+(s,t) \geq 0\} \\
    \Omega_- & = \{(s,t) \in \R^2 \; | \; g_h(s) \geq 0, g_-(s,t) \geq 0\}.
\end{align*}

The degree-$k$ moment truncations of \eqref{eq:risk_mean_meas} from Problem \ref{prob:risk_mean_meas} will be posed in terms by forming finite-vector pseudomoment sequences $y$ as $(\mu_0, \mu_\tau, \mu_+, \mu_-) \rightarrow (y^0, y^\tau, y^+, y^-)$. The restriction of the Liouville relation in \eqref{eq:risk_mean_meas_flow} as parameterized by $\alpha \in \N, \ \beta \in \N, \gamma \in \N^n$ is:
\begin{align}
    \textrm{Liou}_{\alpha \beta \gamma} = L_{y^\tau}&(t^{\alpha+\beta} x^\gamma) - L_{y^0}(s^\alpha \delta_{\beta=0} x^\gamma) \nonumber \\
    &- L_{y^+ + y^-}(\hat{\Lie}(s^\alpha t^\beta x^\gamma)). \label{eq:liou_mom}
\end{align}

The degree-$k$ truncation of the mean-type Problem \ref{prob:risk_mean_meas} is
\begin{prob} \label{prob:risk_mean_mom}
    %Find $y^0 \in \R^{\N^{n+1}_{2k}}$, $y^\tau,y^+,y^- \in \R^{\N^{n+2}_{2k}}$ to maximize
    %\begin{subequations} \label{eq:risk_mean_mom}
    %    \begin{align}
    %        p^\star_k = &\ \max\limits_{y^0,y^\tau,y^+,y^-} L_{y^+}(p)/h \label{eq:risk_mean_mom_obj} \\
  %\textrm{s.t.} \   & L_{y^\tau}(v) = L_{y^0}(v|_{t=0}) + L_{(y^- + y^+)}(\Lie v), \label{eq:risk_mean_mom_flow}\\
  %& \quad \forall v(s,t,x) = s^\alpha t^\beta x^\gamma \in \R_{2\tilde{k}}[s,t,x] \notag \\
    %& L_{y^0}(1) = y_{00} = 1 \label{eq:risk_mean_mom_prob}\\
    %& L_{y^+}(1) = y_{+0} = h \label{eq:risk_mean_mom_window} \\
    %& M_k(y_\bullet) \succeq 0, \quad M_k(g_h y_\bullet) \succeq 0, \label{eq:risk_mean_mom_lmi} \\
    %& \quad \forall \bullet \in \{0,\tau,+,-\} \notag \\
    %& M_k(g_- y^-) \succeq 0, \quad M_k(g_+ y^+) \succeq 0 \\
    %& M_k(g_\tau y^\tau) = 0 \label{eq:risk_mean_mom_eq}
    %    \end{align}
    %\end{subequations}
    Find pseudo-moment sequences $y^0,y^\tau \in \R^{\N^{n+1}_{2k}}$, $y^+,y^- \in \R^{\N^{n+2}_{2k}}$ to maximize
    \begin{subequations} \label{eq:risk_mean_mom}
        \begin{align}
            p^\star_k = &\ \max\limits_{y^0,y^\tau,y^+,y^-} L_{y^+}(p)/h \label{eq:risk_mean_mom_obj} \\
  \textrm{s.t.} \   & \textrm{Liou}_{\alpha \beta \gamma}(y^0, y^\tau, y^+, y^-) = 0 \label{eq:risk_mean_mom_flow}\\
  & \quad \forall (\alpha, \beta, \gamma) \in \N^n \notag \\
    & L_{y^0}(1) = y_{0}^0 = 1 \label{eq:risk_mean_mom_prob}\\
    & L_{y^+}(1) = y_{0}^+ = h \label{eq:risk_mean_mom_window} \\
    & M_k(1 \ y^\bullet) \succeq 0, \quad M_k(g_h y^\bullet) \succeq 0, \label{eq:risk_mean_mom_lmi} \\
    & \quad \forall \bullet \in \{0,\tau,+,-\} \notag \\
    & M_k(g_- y^-) \succeq 0, \quad M_k(g_+ y^+) \succeq 0. \label{eq:risk_mean_mom_omega}
    %& M_k(g_\tau y^\tau) = 0 \label{eq:risk_mean_mom_eq}
        \end{align}
    \end{subequations}
\end{prob}
%where constraint \eqref{eq:risk_mean_mom_eq} is obtained by noticing that the equality constraint in $\diag{[h,T]}$ can be replaced by a double inequality constraint $\pm g_\tau(s,t) \geq 0$, which in turn translates, in Lemma~\ref{lem:Psatz} into asking that $Q_{g_\tau y}$ is both positive semidefinite and negative semidefinite, i.e. null.

% It is %also 
% worth noticing that constraint~\eqref{eq:risk_mean_mom_flow} is indexed by $\alpha,\beta \in \N$, $\gamma \in \N^n$ such that $\alpha + \beta + |\gamma| \leq 2\Tilde{k}$, i.e. consists of $\binom{n+2+2\Tilde{k}}{2\Tilde{k}} < \infty$ scalar linear equality constraints; 
The degree-truncated Liouville relation in \eqref{eq:risk_mean_mom_flow} is defined by $\binom{n+2+2\Tilde{k}}{2\Tilde{k}} < \infty$ scalar linear equality constraints.
Constraints~\eqref{eq:risk_mean_mom_prob} and~\eqref{eq:risk_mean_mom_window} are scalar linear equality constraints, and the \ac{LMI} %linear matrix 
constraints~\eqref{eq:risk_mean_mom_lmi},~\eqref{eq:risk_mean_mom_omega} %--\eqref{eq:risk_mean_mom_eq} 
have size at most $\binom{n+2+\tilde{k}}{\tilde{k}} < \infty$. % no 2\tilde{k} here, the moment matrix size is the dimension of the space of polynomials p such that p^2 is sent to a positive number through the localizing matrices, so its degree is half the available degree in the moment vector
The finite-dimensional Problem~\ref{prob:risk_mean_mom} 
may therefore be solved by \ac{SDP} optimization methods, such as interior-point programs.

}
The degree-$k$ truncation to the \ac{ES}-type problem \ref{prob:risk_cvar_meas} requires the following  polynomial:
\begin{align}
    g_p(q) &= (q-p_{\min})(p_{\max} - q).    
\end{align}
Pseudo-moment sequences $y^\nu, {y}^{\hat{\nu}}$ will be used to describe the measures $\nu$ and $\hat{\nu}$ respectively. The following operator can be defined to implement constraint \eqref{eq:risk_cvar_meas_cvarcon} for each $\ell \in \N$:
\begin{align}
    \mathcal{E}_\ell(y^+, y^{\nu}, y^{\hat{\nu}}) = L_{y^+}(p(x)^\ell)/h - \epsilon L_{y^\nu}(q^\ell) - \epsilon L_{y^{\hat{\nu}}}(q^\ell).
\end{align}

Defining the \ac{ES} degree $\delta$ as as $\delta = \lfloor \tilde{k}/\deg p \rfloor$,
the statement of the degree-$k$ truncation Problem \ref{prob:risk_cvar_meas} is:

\begin{prob} \label{prob:risk_cvar_mom}
    %Find $y^0 \in \R^{\N^{n+1}_{2k}}$, $y^\tau,y^+,y^- \in \R^{\N^{n+2}_{2k}}$ to maximize
    %\begin{subequations} \label{eq:risk_mean_mom}
    %    \begin{align}
    %        p^\star_k = &\ \max\limits_{y^0,y^\tau,y^+,y^-} L_{y^+}(p)/h \label{eq:risk_mean_mom_obj} \\
  %\textrm{s.t.} \   & L_{y^\tau}(v) = L_{y^0}(v|_{t=0}) + L_{(y^- + y^+)}(\Lie v), \label{eq:risk_mean_mom_flow}\\
  %& \quad \forall v(s,t,x) = s^\alpha t^\beta x^\gamma \in \R_{2\tilde{k}}[s,t,x] \notag \\
    %& L_{y^0}(1) = y_{00} = 1 \label{eq:risk_mean_mom_prob}\\
    %& L_{y^+}(1) = y_{+0} = h \label{eq:risk_mean_mom_window} \\
    %& M_k(y_\bullet) \succeq 0, \quad M_k(g_h y_\bullet) \succeq 0, \label{eq:risk_mean_mom_lmi} \\
    %& \quad \forall \bullet \in \{0,\tau,+,-\} \notag \\
    %& M_k(g_- y^-) \succeq 0, \quad M_k(g_+ y^+) \succeq 0 \\
    %& M_k(g_\tau y^\tau) = 0 \label{eq:risk_mean_mom_eq}
    %    \end{align}
    %\end{subequations}
    Find pseudo-moment sequences $y^0,y^\tau \in \R^{\N^{n+1}_{2k}}$, $y^+,y^- \in \R^{\N^{n+2}_{2\delta }}$, $y^\nu, y^{\hat{\nu}} \in \R^{\N^{1}_{2 \delta}}$ to maximize
    \begin{subequations} \label{eq:risk_cvar_mom}
        \begin{align}
            p^\star_{ES, k} = &\ \max\limits_{y^0,y^\tau,y^+,y^-, y^{\nu}, y^{\hat{\nu}}} L_{y^+}(q) \label{eq:risk_cvar_mom_obj} \\
  \textrm{s.t.} \quad  &   \textrm{Constraints \eqref{eq:risk_mean_mom_flow}-\eqref{eq:risk_mean_mom_omega}} \label{eq:risk_cvar_mom_inherit} \\
  & \mathcal{E}_{\ell}(y^+, y^{\nu}, y^{\hat{\nu}}) = 0 \label{eq:risk_mean_mom_objraise} \quad \forall \ell \in \N\\
    & L_{y^{\nu}}(1) = y_{0}^\nu = 1 \label{eq:risk_cvar_mom_prob}\\
    & M_\delta(1 \ y^\bullet) \succeq 0, \quad M_\delta(g_p y^\bullet) \succeq 0, \label{eq:risk_cvar_mom_lmi} \\
    & \quad \forall \bullet \in \{\nu, \hat{\nu}\}. 
    %& M_k(g_\tau y^\tau) = 0 \label{eq:risk_mean_mom_eq}
        \end{align}
    \end{subequations}
\end{prob}
 
\begin{thm}
\label{thm:convergence_mean}
    When Assumptions A1-A6 are imposed, 
 the objectives of \eqref{eq:risk_mean_mom} will satisfy $p^*_k \geq p^*_{k+1} \geq p^*_{k+2} \ldots $ as $\lim_{k\rightarrow \infty} p^*_k = p^*$ from Problem \ref{prob:risk_mean_meas}.  Similarly, solutions to Problem \ref{prob:risk_cvar_mom} will satisfy $\lim_{k \rightarrow \infty} p^*_{ES, k} = p^*_{ES}$ from Problem \ref{prob:risk_cvar_meas}.
\end{thm}
\begin{proof}
See Appendix \ref{app:convergence_mean}.
\end{proof}

\subsection{Computational Complexity}

% One approach to solve the \ac{SOS} program in Problem \eqref{eq:risk_mean_sos} is to cast it as an \ac{SDP} and to employ primal-dual interior point methods. 
\old{
The per-iteration complexity of solving an \ac{SDP} derived from the degree-$k$ Moment-\ac{SOS} hierarchy in $n$ variables is $O(n^{6k})$ and $O(k^{4n})$ \cite{lasserre2009moments}. 
This scaling is due to the $\binom{n+k}{k}$ size of \ac{PSD} localizing matrix constraints. Table \ref{tab:complexity_mean} reports the size of the maximal-size  \ac{PSD} matrix constraints (for each pseudo-moment sequence) from Problem \ref{prob:risk_mean_mom}.}
Table \ref{tab:complexity_cvar} reports the additional maximal \ac{PSD} matrix sizes needed to represent moment matrices of $y^\nu$ and $y^{\hat{\nu}}$ in the $\ac{ES}$ case.
\old{The scaling of solving Problem \ref{prob:risk_mean_mom} or \ref{prob:risk_cvar_mom} by interior-point methods will therefore grow in a jointly polynomial manner as  $(n+2)^{6 \tilde{k}}$ or $\tilde{k}^{4(n+2)}$.
% One of the main indicators of increase in computational complexity is in the growing size of the Gram matrices associated with \ac{SOS} constraints. 
% \urg{adapt the table to the moment formulation. include actual computational complexity of the interior point algo.}

\begin{table}[h]
    \centering
    \caption{\label{tab:complexity_mean}Size of \ac{PSD} matrices needed to represent formulation \eqref{eq:risk_mean_mom} of Problem \ref{prob:risk_mean_mom} at degree $k$}
    \begin{tabular}{r|c c c c}

        Measure & $\mu_0$ & $\mu_\tau$ & $\mu_+$ & $\mu_-$\\
        Constraint & $M_k(y^0)$ & $M_k(y^\tau)$ & $M_{\tilde{k}}(y^+)$ & $M_{\tilde{k}}(y^-)$\\
         \ac{PSD}  Size & $\binom{n+1+k}{k}$ & $\binom{n+1+k}{k}$ & $\binom{n+2+\tilde{k}}{\tilde{k}}$& $\binom{n+2+\tilde{k}}{\tilde{k}}$
    \end{tabular}    

\end{table}
}

 \begin{table}[h]\caption{\label{tab:complexity_cvar}Size of additional \ac{PSD} matrices needed to represent formulation \eqref{eq:risk_cvar_mom} of Problem \ref{prob:risk_cvar_mom} at degree $k$}
 \centering
    \begin{tabular}{r|c c}
        Measure & $\nu$ & $\hat{\nu}$\\
        Constraint & $M_k(y^{\nu})$ & $M_k(y^{\hat{\nu}})$ \\
         \ac{PSD}  Size & $1+\delta$ & $1+\delta$
    \end{tabular}    

\end{table}

\begin{rmk}
    Consider a discrete-time system with $T/\Delta T$ steps and time window $h$. The time-windowed risk analysis problem in discrete-time can also be implemented as $(T-h)/\Delta$ separate stage-cost maximization problems over a rolling time window. Each of these stage-cost problems will have maximal-size \ac{PSD} matrix $M_{\tilde{k}}(y^+) = \binom{n+1+\tilde{k}}{\tilde{k}}$, as compared to the size $M_{\tilde{k}}(y^+) = \binom{n+2+\tilde{k}}{\tilde{k}}$ in Table \ref{tab:complexity_mean}. While such a multiple-instance solution approach could be a desirable alternative in the discrete-time case, unfortunately no such recourse (finding lower-size \ac{PSD} matrices) exists in the proposed continuous-time formulation.
\end{rmk}

% \urg{The measures $\mu_-$ and $\mu_+$ have $(n+2)$ variables. The scaling will therefore grow as $\binom{n+2+k}{k}$.}
\section{Numerical Examples}

\label{sec:examples}
\old{
MATLAB (2023b) code to generate all examples is publicly available\footnote{\url{https://doi.org/10.3929/ethz-b-000662948}}.
% MATLAB (2021a) code to generate the below examples is publicly available at
% \url{https://github.com/daishuyu/noise-in-observations}. 
Dependencies include Gloptipoly \cite{henrion2003gloptipoly},  YALMIP \cite{lofberg2004yalmip}, and Mosek \cite{mosek92}.

\subsection{Three-Dimensional Twist System}

The Twist system from \cite{miller2023distance} is
\begin{align}
    \label{eq:twist_dynamics}
    \dot{x}_i(t) &= \textstyle \sum_j A_{ij} x_j - B_{ij}(4x_j^3 - 3x_j)/2,
    \end{align}
    \begin{align}
    \label{eq:twist_parameters}
    A &= \begin{bmatrix}-1 & 1 & 1\\ -1 &0 &-1\\ 0 & 1 &-2\end{bmatrix} &  B &=  \begin{bmatrix}-1 & 0 & -1\\ 0 &1 &1\\ 1 & 1 &0\end{bmatrix}.
       \end{align}

This example finding bounding the time-averaged value of $p(x) = x_3$ for parameters of $X = [-1, 1]^3, \ h = 1, \ T=5.$ Trajectories start in the initial set $X_0 = \{x \in \R^3 \mid (x_1 +0.75)^2 + (x_2- 0.4)^2 + (x_3+0.1)^2 \leq 0.2^2\}.$ Figure \ref{fig:twist_state} plots trajectories of \eqref{eq:twist_dynamics} in cyan using the same coloring convention as in Figure \ref{fig:twist_state}. Solving Problem \ref{prob:risk_mean_mom} at degrees $k\in1..4$ results in the bounds $p^*_{1:4} = [0.4687, 0.3497, 0.3475, 0.3470]$. At the degree $k=4$ relaxation, the approximately-optimal parameters of $x_0^* \approx [ -0.8705, 0.5586,-0.1027]$ (blue circle), $x_p^* \approx [    0.4593, 0.2197, 0.3043]$ (blue star), and $t_p^* \approx [     1.8315]$  are recovered. 

\begin{figure}
    \centering
    \includegraphics[width=0.9\exfiglength]{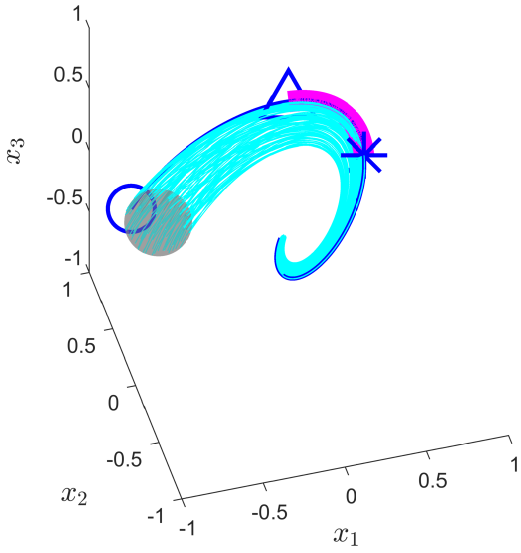}
    \caption{Trajectories of Twist system \eqref{eq:twist_dynamics} starting from the gray-sphere initial set $X_0$.}
    \label{fig:twist_state}
\end{figure}

Figure \ref{fig:twist_cost} plots the time-windowed values of $p$ along trajectories, and highlights the near-optimal trajectory.

\begin{figure}
    \centering
    \includegraphics[width=\exfiglength]{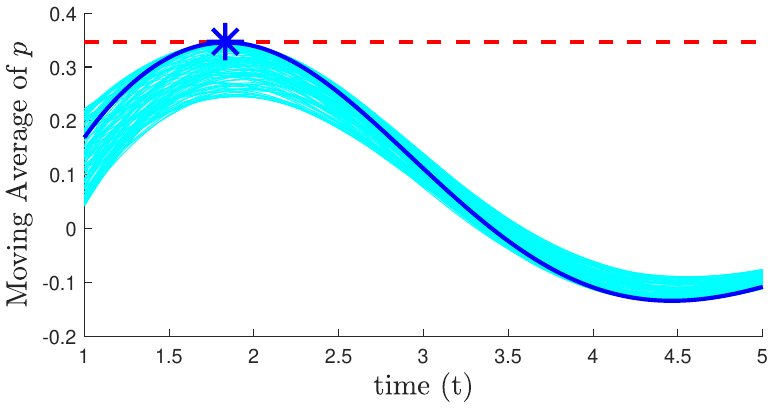}
    \caption{Evolution of time-windowed mean of $p(x)=x_3, h=1$ along trajectories in Figure \ref{fig:twist_state}. }
    \label{fig:twist_cost}
\end{figure}

\subsection{Two-Dimensional Decaying Oscillator}

This example involves a nonlinear decaying oscillator system. It will begin with deterministic \ac{ODE} dynamics, and will advance to \ac{SDE} dynamics.

\subsubsection{Deterinistic Decaying Oscillator}
\label{sec:decay_ode}
\begin{align}
    \dot{x} = \begin{bmatrix}
        -0.5x_1 + 3x_2 (x_1-0.1)^2 \\ -3 x_1 -0.5 x_2 + (x_2 - 0.1)^2 + 1.5
    \end{bmatrix}. \label{eq:osc}
\end{align}

The parameters of this problem are $X = [-0.5, 2.5] \times [-2, 1.5],$ $T = 5, $ $h=1.5, $ and $X_0 = \{x \in \R^2 \mid x_1^2 +(x_2-0.7)^2 \leq 0.1^2\}$. Upper-bounds on the  maximal time-averaged value of $p(x) = x_2$ as found through solving Problem \ref{prob:risk_mean_mom} are $p^*_{1:5} = \{1.5, 0.6376, 0.3852, 0.2741, 0.2433\}.$ Figures \ref{fig:osc_mean_state} and \ref{fig:osc_mean_cost} plot sampled trajectories and evaluations of the time-windowed cost $\int_{t'=t-1.5}^{t} p(x(t' \mid x_0)) dt'$, respectively. Even though near-optimal values of $(x_0^*, t_p^*, x_p^*)$ are not recovered from the numerical solution of Problem \ref{prob:risk_mean_mom} (because the solved moment matrices $M_5(1 y^0), \ M_5(1 y^\tau)$ have a second-largest eigenvalue above the $10^{-3}$ threshold when compared to a top eigenvalue of 1.999), the $k=5$ upper-bound as plotted in Figure \ref{fig:osc_mean_cost} remains an upper-bound for the time-windowed mean objective of sampled trajectories. Note how the instantaneous peak of $\approx  1.01$ is much greater than the $h=1.5$ time-windowed upper bound of 0.2433. In this system, the time-windowed maximum occurs on the second pass about the oscillator (with a crest around time $t=3.2$ in Figure \ref{fig:osc_mean_cost}).

\begin{figure}[t]
    \centering
    \includegraphics[width=\exfiglength]{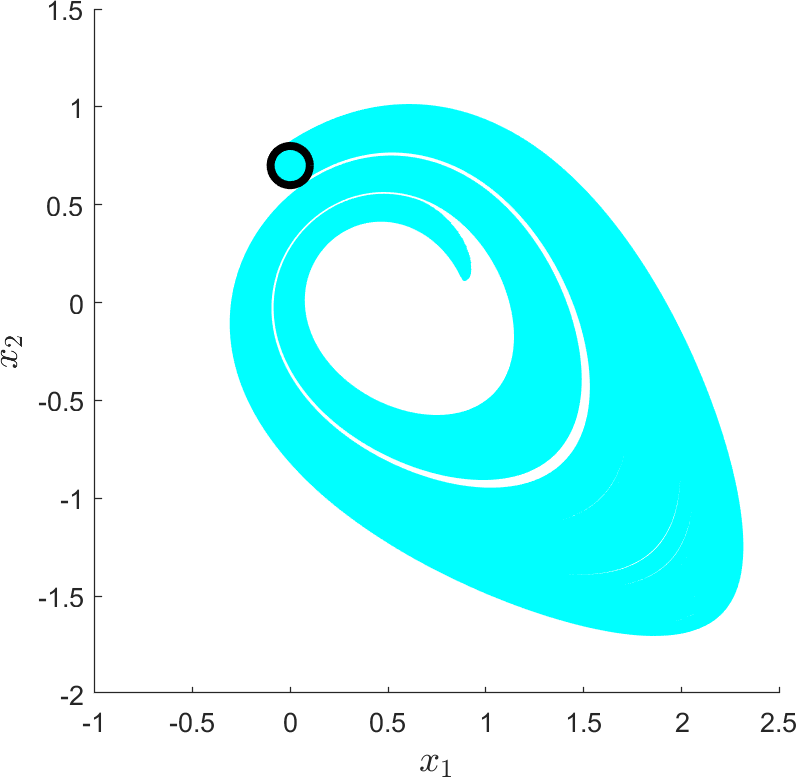}
    \caption{Trajectories of the oscillator \eqref{eq:osc} starting from the black-circle initial set $X_0$.}
    \label{fig:osc_mean_state}
\end{figure}

{\color{black}We now consider the \ac{ES}-type time-windowed risk analysis problem with parameter $\epsilon = 0.15$. The  moment-\ac{SOS} truncation program \ref{prob:risk_cvar_mom} for \ac{ES}-type risk analysis at degrees $k=1:5$  yields the bounds $p^*{ES, 1:5} = [1.5, 1.5, 1.1949, 1.0446, 1.0073]$. The maximal $0.15$-level \ac{ES} as empirically found through sampling 1000 trajectories is 0.9921, and the maximal pointwise value of $p$ under this sampling is 1.0109.
Figure \ref{fig:osc_cvar_cost} plots the time-windowed \ac{ES} of sampled trajectories (cyan) as well as the $p^*_{ES, 5}$ upper bound (red).}

\begin{figure}[h]
    \centering
    \includegraphics[width=\exfiglength]{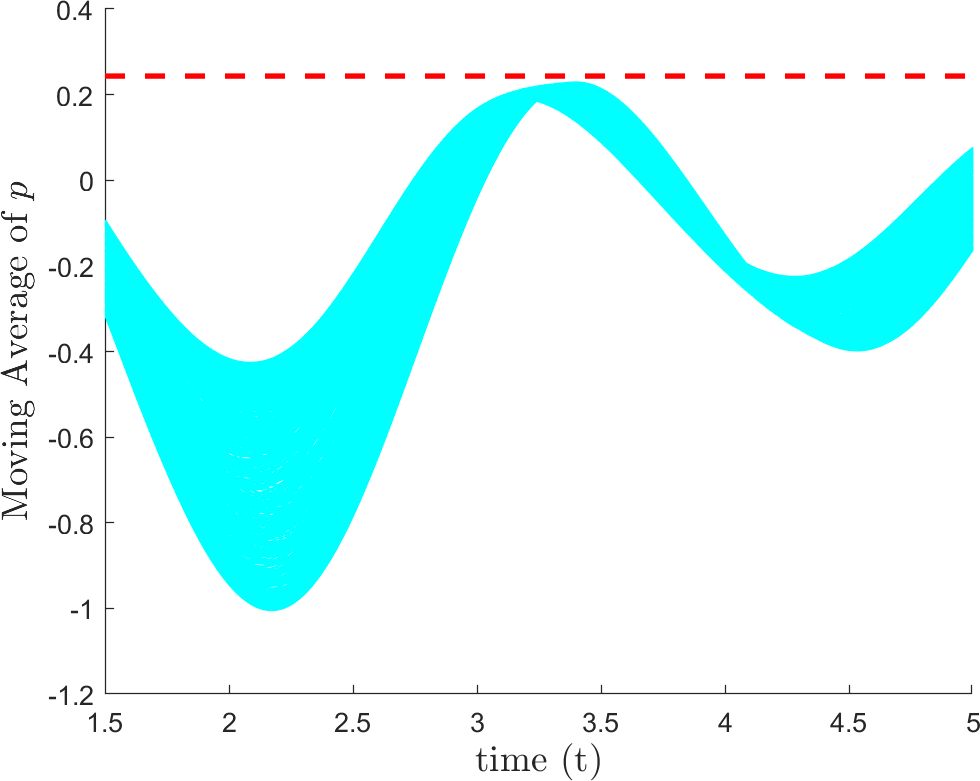}
    \caption{Time-windowed mean of $p(x) = x_2, \ h = 1.5$ along trajectories in Figure \ref{fig:osc_mean_state}, with the $k=5$ bound plotted as the red dotted line. }
    \label{fig:osc_mean_cost}
\end{figure}
}

\begin{figure}[h]
    \centering
    \includegraphics[width=\exfiglength]{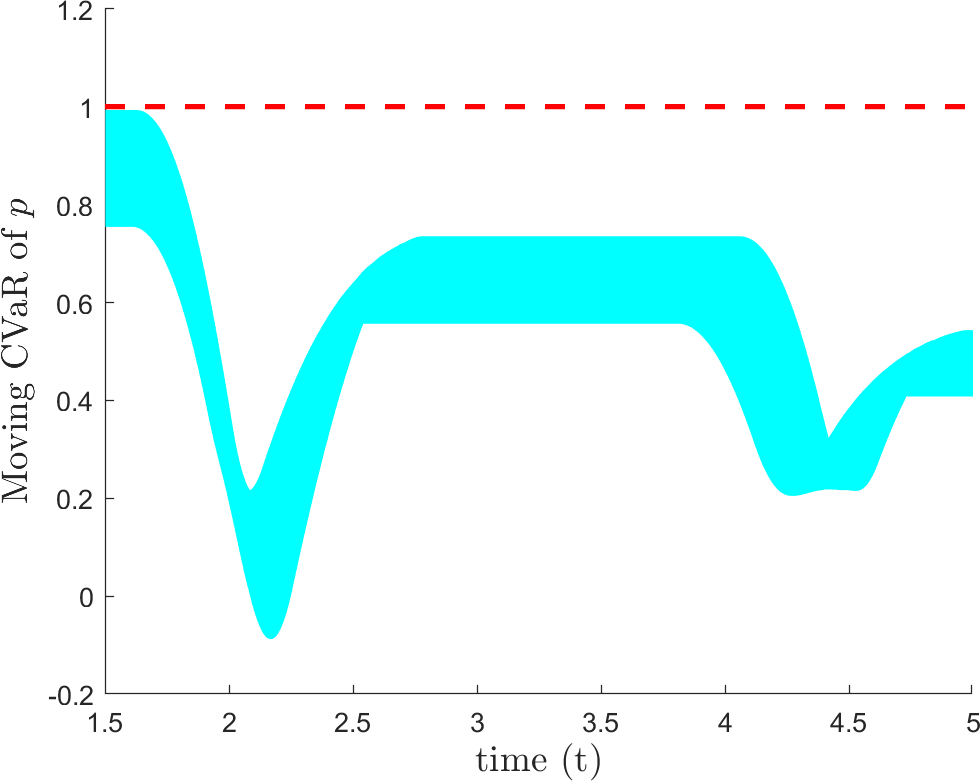}
    \caption{Time-windowed \ac{ES} of $p(x) = x_2, \ h = 1.5, \ \epsilon = 0.15$ along trajectories in Figure \ref{fig:osc_mean_state}, with the $k=5$ bound plotted as the red dotted line. }
    \label{fig:osc_cvar_cost}
\end{figure}

\subsubsection{Stochastic Decaying Oscillator}

We modify the deterministic \ac{ODE} dynamics in \eqref{eq:osc} to become the following \ac{SDE}:
\begin{align}
    dx = \begin{bmatrix}
        -0.5x_1 + 3x_2 (x_1-0.1)^2 \\ -3 x_1 -0.5 x_2 + (x_2 - 0.1)^2 + 1.5
    \end{bmatrix} dt + \begin{bmatrix}
        0.1 \\ 0
    \end{bmatrix}dW. \label{eq:osc_stoch}
\end{align}

Risk analysis of \eqref{eq:osc_stoch} will proceed with the same parameters $(X, T, h, X_0, p)$ as used in the prior section \ref{sec:decay_ode}.
Trajectories of \eqref{eq:osc_stoch} are plotted in Figure \ref{fig:osc_stoch}. Each of these 100 trajectories had their initial point uniformly sampled from $X_0$, and the integration proceeded by Euler-Maruyama sampling with time increment $5 \times 10^{-3}$. Trajectory execution stopped upon contact with the boundary $\partial X$.
\begin{figure}[h]
    \centering
    \includegraphics[width=\exfiglength]{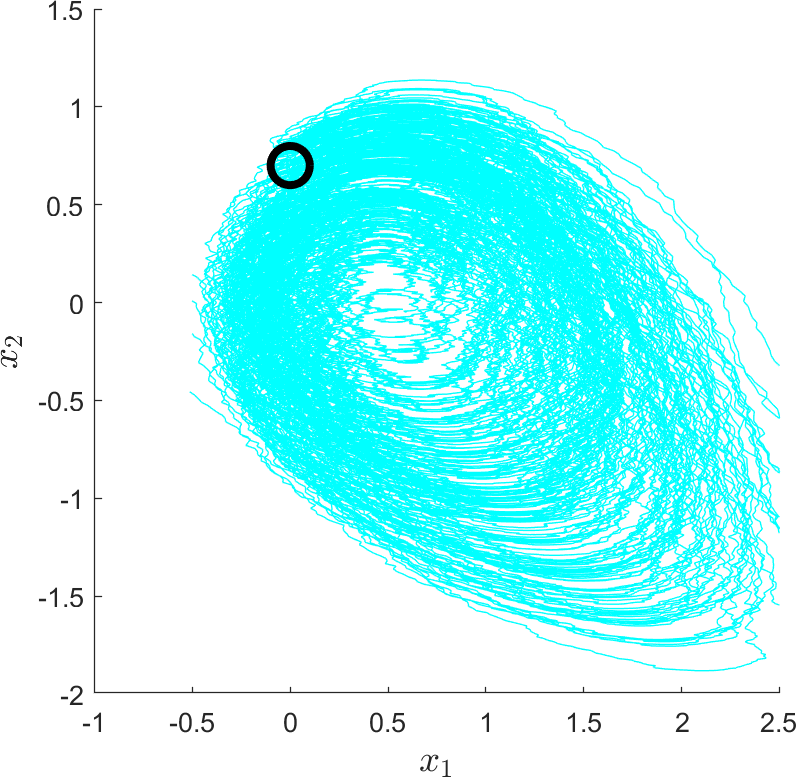}
    \caption{100 sampled trajectories of the stochastic decaying oscillator with dynamics in \eqref{eq:osc_stoch} starting from the initial set $X_0$ (black circle).}
    \label{fig:osc_stoch}
\end{figure}
 Solving the mean-type Problem \ref{prob:risk_mean_mom} at degrees $k=1..5$ produces the bounds $p^*_{1:5} = [1.5, 0.6489, 0.4113, 0.2850, 0.2488].$ 
 
 Figure \ref{fig:osc_stoch_mean} compares the mean-bound $p^*_5$ with an empirically-generated time-windowed trace of time-windowed mean values.
 \begin{figure}[h]
    \centering
    \includegraphics[width=\exfiglength]{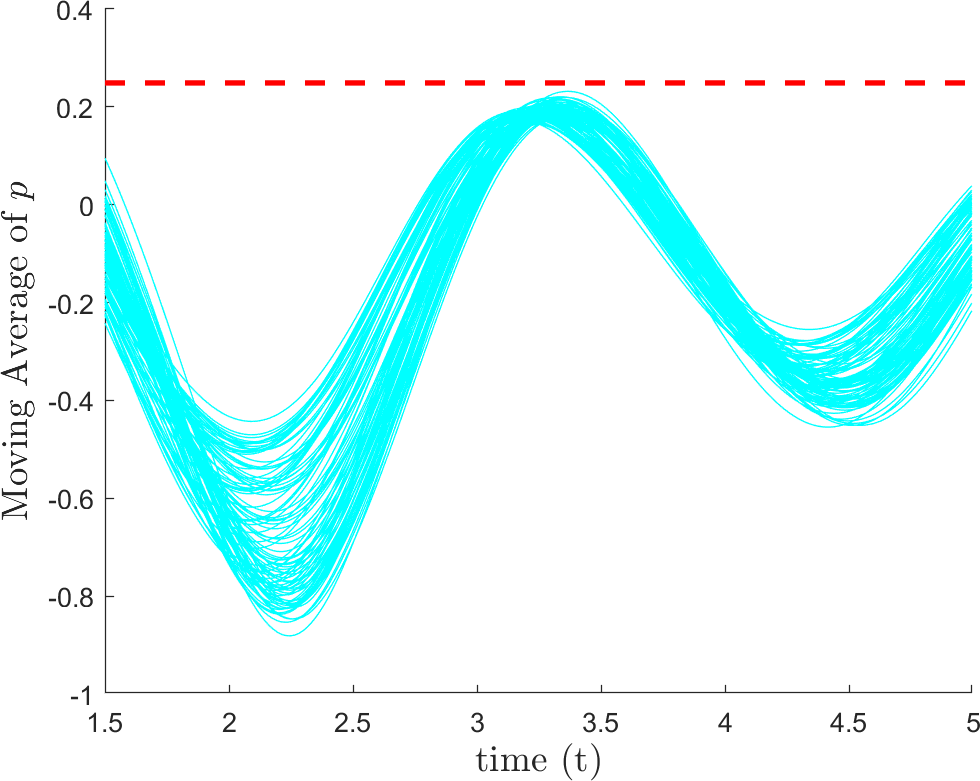}
    \caption{Time-windowed mean bound (red) of $p^*_5$ vs. empirical time-windowed mean (cyan) for system \eqref{eq:osc_stoch}.}
    \label{fig:osc_stoch_mean}
\end{figure}
Solving the \ac{ES}-type Problem \ref{prob:risk_cvar_mom} with $\epsilon=0.15$ at degrees $k=1..4$ yields the bounds $p^*_{ES, 1:4} = [1.5, 1.5, 1.2344, 1.0604]$. The $k=5$ truncation of Problem \ref{prob:risk_cvar_mom} produces an invalid bound of $0.9556$ due to numerical conditioning and error. Figure \ref{fig:osc_stoch_cvar} compares the bound $p^*_{ES, 5}$ with the time-windowed empirical \ac{ES} (under the same 10,000 samples as used in Figure \ref{fig:osc_stoch_mean}).
 \begin{figure}[h]
    \centering
    \includegraphics[width=\exfiglength]{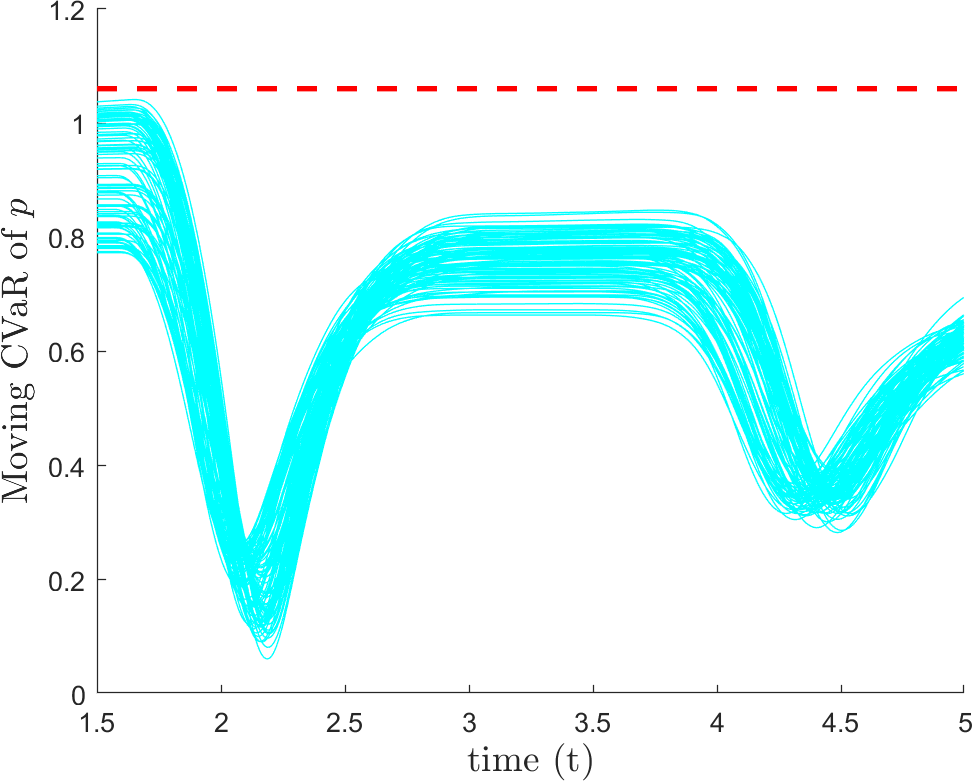}
    \caption{Time-windowed \ac{ES} bound of $p^*_{ES, 4}$ (red) at $\epsilon=0.15$ vs. empirical time-windowed \ac{ES} (cyan) for system \eqref{eq:osc_stoch}.}
    \label{fig:osc_stoch_cvar}
\end{figure}

\section{Conclusion}

\label{sec:conclusion}

% \urg{rewrite this.}

Time-windowed risks of stochastic processes offer a different perspective on trajectory safety as compared to instantaneous risks. This paper upper-bounded the time-windowed mean and \ac{ES} risk of a stochastic process by relaxing Problem \ref{prob:risk_window} into  infinite-dimensional \acp{LP} in occupation measures (Problems \ref{prob:risk_mean_meas}, \ref{prob:risk_mean_cont}, \ref{prob:risk_cvar_meas}, and  \ref{prob:risk_cvar_cont}). Contiguity of the time window was enforced by adding a new stopping-time coordinate $s$ and ensuring that $\mu_\pm$  have temporal supports in the partitioned sets $\Omega_\pm$. The resultant \ac{LP} were truncated through the moment-\ac{SOS} hierarchy of \acp{SDP} to obtain tractable finite-dimensional programs (Problems \ref{prob:risk_mean_mom} and \ref{prob:risk_cvar_mom}). This sequence of \acp{SDP} was proven to converge to the optimal value $P^*$ from \eqref{eq:risk_window} as the relaxation degree tends towards infinity when assumptions A1-A9 hold (Theorem \ref{thm:convergence_mean}). 

Several extensions to this method are possible. The presented scheme is amenable to wider classes of dynamics than the considered Markovian stochastic processes in continuous-time or discrete-time. This broader class of dynamics includes 
non-Markovian stochastic processes \cite{miller2023delay} and hybrid stochastic processes involving continuous-time and discrete-time transitions \cite{miller2023hybrid} (such as Piecewise Markov Decision Processes). The main focus of future work on this topic is solving the time-windowed-risk-minimizing control problem for continuous-time processes by establishment of no-relaxation-gap conditions and efficient numerical computation. The effectiveness of the presented methodology for risk estimation will increase as finite truncations of infinite \acp{LP} become faster and possess greater numerical accuracy. The development of improved truncation methods will therefore increase the feasibility of large-scale time-windowed risk estimation.

\section{Acknowledgements}

The authors would like to thank Riccardo Bonalli, Beno\^{i}t Bonnet-Weill,  Adrian Hauswirth, and Nicolas Lanzetti for discussions about stochastic processes and power systems.

% \urg{Optional acknowledgements.}
% The authors thank Milan Korda for his discussions about occupation measures and time-varying uncertainty.
\bibliographystyle{IEEEtran}
\bibliography{references.bib}

\begin{thebibliography}{10}
\providecommand{\url}[1]{#1}
\csname url@samestyle\endcsname
\providecommand{\newblock}{\relax}
\providecommand{\bibinfo}[2]{#2}
\providecommand{\BIBentrySTDinterwordspacing}{\spaceskip=0pt\relax}
\providecommand{\BIBentryALTinterwordstretchfactor}{4}
\providecommand{\BIBentryALTinterwordspacing}{\spaceskip=\fontdimen2\font plus
\BIBentryALTinterwordstretchfactor\fontdimen3\font minus \fontdimen4\font\relax}
\providecommand{\BIBforeignlanguage}[2]{{%
\expandafter\ifx\csname l@#1\endcsname\relax
\typeout{** WARNING: IEEEtran.bst: No hyphenation pattern has been}%
\typeout{** loaded for the language `#1'. Using the pattern for}%
\typeout{** the default language instead.}%
\else
\language=\csname l@#1\endcsname
\fi
#2}}
\providecommand{\BIBdecl}{\relax}
\BIBdecl

\bibitem{machowski1997power}
J.~Machowski, J.~W. Bialek, J.~Bialek, and J.~R. Bumby, \emph{Power system dynamics and stability}.\hskip 1em plus 0.5em minus 0.4em\relax John Wiley \& Sons, 1997.

\bibitem{turner2010transformer}
R.~A. Turner and K.~S. Smith, ``Transformer inrush currents,'' \emph{IEEE Industry Applications Magazine}, vol.~16, no.~5, pp. 14--19, 2010.

\bibitem{miller2023chancepeak}
J.~Miller, M.~Tacchi, M.~Sznaier, and A.~Jasour, ``{Peak Value-at-Risk Estimation for Stochastic Processes using Occupation Measures},'' 2023, arXiv:2303.16064.

\bibitem{bonalli2023first}
R.~Bonalli and B.~Bonnet, ``First-order pontryagin maximum principle for risk-averse stochastic optimal control problems,'' \emph{SIAM Journal on Control and Optimization}, vol.~61, no.~3, pp. 1881--1909, 2023.

\bibitem{aastrom2012introduction}
K.~J. {\AA}str{\"o}m, \emph{Introduction to stochastic control theory}.\hskip 1em plus 0.5em minus 0.4em\relax Courier Corporation, 2012.

\bibitem{lewis1980relaxation}
R.~Lewis and R.~Vinter, ``{Relaxation of Optimal Control Problems to Equivalent Convex Programs},'' \emph{Journal of Mathematical Analysis and Applications}, vol.~74, no.~2, pp. 475--493, 1980.

\bibitem{altman2021constrained}
E.~Altman, \emph{Constrained Markov decision processes}.\hskip 1em plus 0.5em minus 0.4em\relax Routledge, 2021.

\bibitem{delbaen2000coherent}
F.~Delbaen and S.~Biagini, \emph{Coherent risk measures}.\hskip 1em plus 0.5em minus 0.4em\relax Springer, 2000.

\bibitem{cho2002linear}
M.~J. Cho and R.~H. Stockbridge, ``{Linear Programming Formulation for Optimal Stopping Problems},'' \emph{SIAM J. Control Optim.}, vol.~40, no.~6, pp. 1965--1982, 2002.

\bibitem{kashima2010optimization}
K.~Kashima and R.~Kawai, ``An optimization approach to weak approximation of {L\'e}vy-driven stochastic differential equations,'' in \emph{Perspectives in Mathematical System Theory, Control, and Signal Processing}.\hskip 1em plus 0.5em minus 0.4em\relax Springer, 2010, pp. 263--272.

\bibitem{prajna2004stochastic}
S.~Prajna, A.~Jadbabaie, and G.~J. Pappas, ``{Stochastic Safety Verification Using Barrier Certificates},'' in \emph{2004 43rd IEEE conference on decision and control (CDC)(IEEE Cat. No. 04CH37601)}, vol.~1.\hskip 1em plus 0.5em minus 0.4em\relax IEEE, 2004, pp. 929--934.

\bibitem{prajna2007framework}
------, ``{A Framework for Worst-Case and Stochastic Safety Verification Using Barrier Certificates},'' \emph{IEEE Transactions on Automatic Control}, vol.~52, no.~8, pp. 1415--1428, 2007.

\bibitem{xue2023reach}
B.~Xue, N.~Zhan, and M.~Fr{\"a}nzle, ``Reach-avoid analysis for polynomial stochastic differential equations,'' \emph{IEEE Transactions on Automatic Control}, 2023.

\bibitem{tobasco2018optimal}
I.~Tobasco, D.~Goluskin, and C.~R. Doering, ``Optimal bounds and extremal trajectories for time averages in nonlinear dynamical systems,'' \emph{Physics Letters A}, vol. 382, no.~6, pp. 382--386, 2018.

\bibitem{jasour2019risk}
A.~M. Jasour and B.~C. Williams, ``{Risk Contours Map for Risk Bounded Motion Planning under Perception Uncertainties},'' in \emph{Robotics: Science and Systems}, 2019.

\bibitem{henrion2021moment}
D.~Henrion, M.~Junca, and M.~Velasco, ``{Moment-SOS hierarchy and exit time of stochastic processes},'' \emph{arXiv preprint arXiv:2101.06009}, 2021.

\bibitem{Henrion_2014}
D.~Henrion and M.~Korda, ``{Convex Computation of the Region of Attraction of Polynomial Control Systems},'' \emph{IEEE Trans. Automat. Contr.}, vol.~59, no.~2, p. 297–312, Feb 2014.

\bibitem{majumdar2014convex}
A.~Majumdar, R.~Vasudevan, M.~M. Tobenkin, and R.~Tedrake, ``{Convex Optimization of Nonlinear Feedback Controllers via Occupation Measures},'' \emph{The International Journal of Robotics Research}, vol.~33, no.~9, pp. 1209--1230, 2014.

\bibitem{kariotoglou2013approximate}
N.~Kariotoglou, S.~Summers, T.~Summers, M.~Kamgarpour, and J.~Lygeros, ``Approximate dynamic programming for stochastic reachability,'' in \emph{2013 European Control Conference (ECC)}.\hskip 1em plus 0.5em minus 0.4em\relax IEEE, 2013, pp. 584--589.

\bibitem{schmid2022probabilistic}
N.~Schmid and J.~Lygeros, ``Probabilistic reachability and invariance computation of stochastic systems using linear programming,'' \emph{IFAC-PapersOnLine}, vol.~56, no.~2, pp. 11\,229--11\,234, 2023, 22nd IFAC World Congress.

\bibitem{goluskin2020attractor}
D.~Goluskin, ``Bounding extrema over global attractors using polynomial optimisation,'' \emph{Nonlinearity}, vol.~33, no.~9, p. 4878, 2020.

\bibitem{schlosser2021converging}
C.~Schlosser and M.~Korda, ``Converging outer approximations to global attractors using semidefinite programming,'' \emph{Automatica}, vol. 134, p. 109900, 2021.

\bibitem{oustry2019inner}
A.~Oustry, M.~Tacchi, and D.~Henrion, ``{Inner Approximations of the Maximal Positively Invariant Set for Polynomial Dynamical Systems},'' \emph{IEEE Control Systems Letters}, vol.~3, no.~3, pp. 733--738, 2019.

\bibitem{miller2021uncertain}
J.~Miller, D.~Henrion, M.~Sznaier, and M.~Korda, ``{Peak Estimation for Uncertain and Switched Systems},'' in \emph{2021 60th IEEE Conference on Decision and Control (CDC)}, 2021, pp. 3222--3228.

\bibitem{prajna2004safety}
S.~Prajna and A.~Jadbabaie, ``{Safety Verification of Hybrid Systems Using Barrier Certificates},'' in \emph{International Workshop on Hybrid Systems: Computation and Control}.\hskip 1em plus 0.5em minus 0.4em\relax Springer, 2004, pp. 477--492.

\bibitem{zhao2019optimal}
P.~Zhao, S.~Mohan, and R.~Vasudevan, ``{Optimal Control of Polynomial Hybrid Systems via Convex Relaxations},'' \emph{IEEE Trans. Automat. Contr.}, vol.~65, no.~5, pp. 2062--2077, 2019.

\bibitem{miller2023hybrid}
J.~Miller and M.~Sznaier, ``{Peak Estimation of Hybrid Systems with Convex Optimization},'' 2023, arxiv:2303.11490.

\bibitem{prajna2005methods}
S.~Prajna and A.~Jadbabaie, ``Methods for safety verification of time-delay systems,'' in \emph{Proceedings of the 44th IEEE Conference on Decision and Control}.\hskip 1em plus 0.5em minus 0.4em\relax IEEE, 2005, pp. 4348--4353.

\bibitem{miller2023delay}
J.~Miller, M.~Korda, V.~Magron, and M.~Sznaier, ``Peak estimation of time delay systems using occupation measures,'' in \emph{2023 62nd IEEE Conference on Decision and Control (CDC)}.\hskip 1em plus 0.5em minus 0.4em\relax IEEE, 2023, pp. 5294--5300.

\bibitem{lasserre2009moments}
J.~B. Lasserre, \emph{{Moments, Positive Polynomials And Their} {Applications}}, ser. Imperial College Press Optimization Series.\hskip 1em plus 0.5em minus 0.4em\relax World Scientific Publishing Company, 2009.

\bibitem{fantuzzi2020bounding}
G.~Fantuzzi and D.~Goluskin, ``{Bounding Extreme Events in Nonlinear Dynamics Using Convex Optimization},'' \emph{SIAM Journal on Applied Dynamical Systems}, vol.~19, no.~3, pp. 1823--1864, 2020.

\bibitem{MOHAJERINESFAHANI201643}
P.~{Mohajerin Esfahani}, D.~Chatterjee, and J.~Lygeros, ``The stochastic reach-avoid problem and set characterization for diffusions,'' \emph{Automatica}, vol.~70, pp. 43--56, 2016.

\bibitem{abate2021fossil}
A.~Abate, D.~Ahmed, A.~Edwards, M.~Giacobbe, and A.~Peruffo, ``{FOSSIL: A Software Tool for the Formal Synthesis of Lyapunov Functions and Barrier Certificates using Neural Networks},'' in \emph{Proceedings of the 24th International Conference on Hybrid Systems: Computation and Control}, 2021, pp. 1--11.

\bibitem{chow2015risk}
Y.~Chow, A.~Tamar, S.~Mannor, and M.~Pavone, ``Risk-sensitive and robust decision-making: a cvar optimization approach,'' \emph{Advances in neural information processing systems}, vol.~28, 2015.

\bibitem{chapman2022optimizing}
M.~P. Chapman, M.~Fau{\ss}, and K.~M. Smith, ``On optimizing the conditional value-at-risk of a maximum cost for risk-averse safety analysis,'' \emph{IEEE Transactions on Automatic Control}, 2022.

\bibitem{chapman2021risk}
M.~P. Chapman, R.~Bonalli, K.~M. Smith, I.~Yang, M.~Pavone, and C.~J. Tomlin, ``Risk-sensitive safety analysis using conditional value-at-risk,'' \emph{IEEE Transactions on Automatic Control}, vol.~67, no.~12, pp. 6521--6536, 2021.

\bibitem{ruszczynski2010risk}
A.~Ruszczy{\'n}ski, ``Risk-averse dynamic programming for markov decision processes,'' \emph{Mathematical programming}, vol. 125, pp. 235--261, 2010.

\bibitem{tao2011introduction}
T.~Tao, \emph{{An Introduction to Measure Theory}}.\hskip 1em plus 0.5em minus 0.4em\relax American Mathematical Society Providence, RI, 2011, vol. 126.

\bibitem{artzner1999coherent}
P.~Artzner, F.~Delbaen, J.-M. Eber, and D.~Heath, ``Coherent measures of risk,'' \emph{Mathematical finance}, vol.~9, no.~3, pp. 203--228, 1999.

\bibitem{delbaen2000draft}
F.~Delbaen, ``Draft: Coherent risk measures,'' \emph{Lecture notes, Pisa}, 2000.

\bibitem{rockafellar2002conditional}
R.~T. Rockafellar and S.~Uryasev, ``{Conditional Value-at-Risk for General Loss Distributions},'' \emph{Journal of banking \& finance}, vol.~26, no.~7, pp. 1443--1471, 2002.

\bibitem{follmer2010convex}
H.~F{\"o}llmer and A.~Schied, ``Convex and coherent risk measures,'' \emph{Encyclopedia of Quantitative Finance}, pp. 355--363, 2010.

\bibitem{ahmadi2012entropic}
A.~Ahmadi-Javid, ``Entropic value-at-risk: A new coherent risk measure,'' \emph{Journal of Optimization Theory and Applications}, vol. 155, pp. 1105--1123, 2012.

\bibitem{kou2013external}
S.~Kou, X.~Peng, and C.~C. Heyde, ``External risk measures and basel accords,'' \emph{Mathematics of Operations Research}, vol.~38, no.~3, pp. 393--417, 2013.

\bibitem{oksendal2003stochastic}
B.~{\O}ksendal, ``{Stochastic Differential Equations: An Introduction with Applications },'' in \emph{Stochastic differential equations}.\hskip 1em plus 0.5em minus 0.4em\relax Springer, 2003, pp. 65--84.

\bibitem{henrion2009approximate}
D.~Henrion, J.~B. Lasserre, and C.~Savorgnan, ``{Approximate Volume and Integration for Basic Semialgebraic Sets},'' \emph{SIAM review}, vol.~51, no.~4, pp. 722--743, 2009.

\bibitem{putinar1993compact}
M.~Putinar, ``{Positive Polynomials on Compact Semi-algebraic Sets},'' \emph{Indiana University Mathematics Journal}, vol.~42, no.~3, pp. 969--984, 1993.

\bibitem{henrion2003gloptipoly}
D.~Henrion and J.-B. Lasserre, ``{GloptiPoly: Global Optimization over Polynomials with Matlab and SeDuMi},'' \emph{ACM Transactions on Mathematical Software (TOMS)}, vol.~29, no.~2, pp. 165--194, 2003.

\bibitem{lofberg2004yalmip}
J.~{Lofberg}, ``{YALMIP : a toolbox for modeling and optimization in MATLAB},'' in \emph{ICRA (IEEE Cat. No.04CH37508)}, 2004, pp. 284--289.

\bibitem{mosek92}
M.~ApS, \emph{The MOSEK optimization toolbox for MATLAB manual. Version 10.1.}, 2023.

\bibitem{miller2023distance}
J.~Miller and M.~Sznaier, ``{Bounding the Distance to Unsafe Sets with Convex Optimization},'' \emph{IEEE Transactions on Automatic Control}, pp. 1--15, 2023.

\bibitem{tacchi2021thesis}
M.~Tacchi, ``{Moment-SOS hierarchy for large scale set approximation. Application to power systems transient stability analysis},'' Ph.D. dissertation, Toulouse, INSA, 2021.

\bibitem{tacchi2022convergence}
------, ``Convergence of {L}asserre’s hierarchy: the general case,'' \emph{Optimization Letters}, vol.~16, no.~3, pp. 1015--1033, 2022.

\end{thebibliography}

\appendix
\section{Proof of Mean Strong Duality in Theorem \ref{thm:duality_mean}}
\label{app:duality_mean}
% \urg{Version 2}

% \label{app:duality}
% This proof will follow the method used in \rev{Theorem 2.6 of  \cite{tacchi2021thesis}} to prove duality. 

This proof will follow the conventions and methods from Theorem 2.6 of  \cite{tacchi2021thesis} to prove strong duality of \eqref{eq:risk_mean_meas} and \eqref{eq:risk_mean_cont}. The \ac{LP} \eqref{eq:risk_mean_meas} and \eqref{eq:risk_mean_cont} will be cast into the form of standard conic optimization problems.

\subsection{Weak Duality}

The resident state spaces may be defined as
\begin{align}
    \mathcal{X}' &= C([h, T] \times X_0) \times C([h, T] \times X)^2  \nonumber  \\
    & \qquad \times C(\Omega_+  \times X)\times C(\Omega_-  \times X)\label{eq:res_spaces}\\
    \mathcal{X} &= \mathcal{M}([h, T] \times X_0) \times \mathcal{M}([h, T] \times X)^2   \nonumber \\
    & \qquad \times \mathcal{M}(\Omega_+  \times X)\times \mathcal{M}(\Omega_-  \times X). \nonumber
\end{align}

Nonnegative subcones of the spaces in \eqref{eq:res_spaces} are
\begin{align}
    \mathcal{X}'_+ &= C_+([h, T] \times X_0) \times C_+([h, T] \times X)^2 \nonumber  \\
    & \qquad \times C_+(\Omega_+  \times X)\times C_+(\Omega_-  \times X)\label{eq:dual_spaces}\\
    \mathcal{X}_+ &= \mathcal{M}_+([h, T] \times X_0) \times \mathcal{M}_+([h, T] \times X)^2 \nonumber  \\
    & \qquad \times \mathcal{M}_+(\Omega_+  \times X)\times \mathcal{M}_+(\Omega_-  \times X). \nonumber
\end{align}

The spaces $(\mathcal{X}_+', \mathcal{X}_+)$ in \eqref{eq:res_spaces} are topological duals under the compactness assumption in A1. The measure variables $\bbmu = (\mu_0, \mu_\tau, \mu_+, \mu_-)$ of \eqref{eq:risk_meas_def} obey the membership $\bbmu \in \mathcal{X}_+.$

The constraint-sets $(\mathcal{Y}', \mathcal{Y})$ are defined as 
\begin{align}
    \mathcal{Y}' &=  \cs([h, T] \times [0, T] \times X) \times \R \times \R \\
    \mathcal{Y} &= \cs([h, T] \times [0, T] \times X)' \times 0 \times 0. \nonumber 
\end{align}
We subsequently define the symbols $\mathcal{Y}'_+ = \mathcal{Y}'$ and $\mathcal{Y}_+ =  \mathcal{Y}$ in accordance with the convention in \cite{tacchi2021thesis}, given that there are no affine inequality constraints present in \eqref{eq:risk_mean_meas}. The space $\mathcal{Y}$ will have a sup-norm-bounded weak topology, and the space $\mathcal{X}$ possesses a weak-* topology. The variables $\bell =  (v, \gamma, \xi)$ from \eqref{eq:risk_mean_cont} satisfy $\bell \in \mathcal{Y}'.$

An affine operator $\A: \mathcal{X} \rightarrow \mathcal{Y}$ induced from \eqref{eq:risk_mean_meas_flow}-\eqref{eq:risk_mean_meas_prob} with adjoint $\A': \mathcal{Y}' \rightarrow \mathcal{X}'$ operates on variables $(\bbmu, \bell)$ as in:
\begin{align}
    \A(\bbmu) =[&\varphi_\# \mu_\tau - \delta_0 \otimes \mu_0 - \hat{\Lie}^\dagger (\mu_+ + \mu_-), \ \inp{1}{\mu_0}, \ \inp{1}{\mu_+}] \nonumber\\ 
    \A'(\bell) = [&\gamma - v(s,0,x), v(s,s,x), \xi -\hat{\Lie}_f v, -\hat{\Lie}_f v].
\end{align}

The cost and constraint vector for problem \eqref{eq:risk_mean_meas} are
\begin{align}
    {\mathbf{c}} &= [0, 0, p(x)/h, 0], &
    {\mathbf{b}} &= [0, 1, h].
\intertext{These vectors satisfy the pairings of }
\inp{\mathbf{c}}{\bbmu} &= \inp{p}{\mu_+}/h, & \inp{\mathbf{b}}{\bell} &= \gamma + h \xi.
\end{align}

Problem \eqref{eq:risk_mean_meas} can therefore be cast into the standard-form expression
\begin{align}
    p^* =& \sup_{\bbmu \in \mathcal{X}_+} \inp{\mathbf{c}}{\boldsymbol{\mu}} & & \mathbf{b} - \A(\boldsymbol{\mu}) \in \mathcal{Y}_+. \label{eq:risk_mean_meas_std}\\
\intertext{A similar process results in the formulation of \eqref{eq:risk_mean_cont} as}
    d^* = &\inf_{\bell \in \mathcal{Y}_+'} \inp{\boldsymbol{\ell}}{\mathbf{b}}
    & &\A'(\bell) - \mathbf{c} \in \mathcal{X}_+. \label{eq:risk_mean_cont_std}
\end{align}
This transformation ensures weak duality.
\subsection{Strong Duality}
\label{app:strong_duality_mean_sec}

To prove strong duality, we first require the following lemma about boundedness of measure solutions:
\begin{lem}
\label{lem:mean_mass}
   All measures that are feasible for \eqref{eq:risk_mean_meas} will be bounded under A1-A4.
\end{lem}
\begin{proof}
   Boundedness will be proven by the sufficient conditions of compact support and finite mass. Assumption A1 enforces compact support of all measures in \eqref{eq:risk_meas_def}. The initial measure $\mu_0$ has mass 1 by constraint \eqref{eq:risk_mean_meas_prob}, and the forward measure $\mu_+$ has mass $h$ by constraint \eqref{eq:risk_mean_meas_window}. Passing a test function $v(s, t, x) = 1$ through \eqref{eq:risk_mean_meas_flow} results in $\inp{1}{\mu_\tau} = \inp{1}{\mu_0} = 1$. Sending $v(s, t, x) = t$ through \eqref{eq:risk_mean_meas_flow} yields $\inp{t}{\mu_\tau} = \inp{1}{\mu_+ + \mu_-}$. Given that $\inp{t}{\mu_\tau} \in [h, T]$ and that $\mu_-$ is a nonnegative measures with a nonnegative mass, it holds that $\inp{1}{\mu_-} \in [0, \max(0, \inp{t}{\mu_\tau} - \inp{1}{\mu_+})]$. Boundedness of measures is therefore confirmed.
\end{proof}

Strong duality will be proven by fulfillment of the following sufficient conditions \cite[Theorem 2.6]{tacchi2021thesis}:
\begin{enumerate}
    \item[R1] All support sets are compact.
    \item[R2] All feasible $\bbmu$ have finite mass.
    \item[R3] The problem data $(\A, \mathbf{b}, \mathbf{c})$ are defined  by continuous functions in $(s, t, x)$.
    \item[R4] A feasible $\boldsymbol{\mu}_{\textrm{feas}} \in \mathcal{X}_+$ exists with  $\mathbf{b} - \A(\boldsymbol{\mu}_\textrm{feas}) \in \mathcal{Y}_+$.
\end{enumerate}

Requirement R1 is ensured by assumption A1. Requirement R2 is proven by Lemma \ref{lem:mean_mass}. For requirement R3, the vector $\mathbf{b}$ is constant in $(s, t, x)$, and the function $p(x)$ in $\mathbf{c}$ is continuous by A2. Given that the image of $\hat{\Lie}$ is contained in $C([h, T] \times [0, T] \times X)$ (A4-A6), R3 is fulfilled. A feasible $\bbmu$ may be created through the construction procedure from the proof of Theorem \ref{thm:risk_mean_upper}, satisfying R4. As all conditions are met, strong duality is proven.

\section{Proof of ES Strong Duality in Theorem \ref{thm:duality_cvar}}
\label{app:duality_cvar}

The proof of strong duality for Theorem \ref{thm:duality_cvar} echoes the proof of Theorem \ref{thm:duality_mean} outlined above in Appendix \ref{app:duality_mean}. We therefore sketch the highlights and modifications done in the \ac{ES} case.

\subsection{Weak Duality}

The set $\mathcal{X}'$ from \eqref{eq:res_spaces} in the \ac{ES} case has the expression of  
\begin{align}
    \mathcal{X}' &= C([h, T] \times X_0) \times C([h, T] \times X)^2 \\
     & \qquad \times C(\Omega_+  \times X)\times C(\Omega_-  \times X) \nonumber \\
     & \qquad \times C([p_{\min}, p_{\max}])^2. \nonumber 
\end{align}

The spaces $\mathcal{X}, \mathcal{X}'_+, \mathcal{X}_+$ are similarly enriched by continuous functions or measures supported over the set $[p_{\min}, p_{\max}]$. In the \ac{ES} context, the measure variable $\bbmu$ will refer to $\bbmu = (\mu_0, \mu_\tau, \mu_+, \mu_-, \nu, \hat{\nu}).$

The constraint sets in the \ac{ES} case are 
% \begin{subequations}
\begin{align}
    \mathcal{Y}'_+ &=  \cs([h, T] \times [0, T] \times X) \times \R^3 \times C([p_{\min}, p_{\max}]) \\
    \mathcal{Y}_+ &= \cs([h, T] \times [0, T] \times X)' \times 0^3 \times \mathcal{M}([p_{\min}, p_{\max}]), \nonumber
\end{align}
% \end{subequations}
with variables $\bell = (v, \gamma, \xi, \beta, w)$ satisfying $\bell \in \mathcal{Y}_+.$

The affine operator $\mathcal{A}: X \rightarrow \mathcal{Y}$ in the \ac{ES} case is
\begin{align}
    \A(\bbmu) =[&\varphi_\# \mu_\tau - \delta_0 \otimes \mu_0 - \hat{\Lie}^\dagger (\mu_+ + \mu_-), \ \inp{1}{\mu_0}, \ \inp{1}{\mu_+} \nonumber\\
    & \inp{1}{\nu},   \epsilon\nu + \hat{\nu} -p_\# \pi_\#^x \mu_\tau/h] \nonumber\\
    \A'(\bell) = [&\gamma - v(s,0,x), v(s,s,x), \xi - w(p(x))/h-\hat{\Lie}_f v,  \nonumber\\
    & -\hat{\Lie}_f v, \ \beta + \epsilon w(q), \ w(q)].
\end{align}

The cost and constraint for the \ac{ES} Problem \ref{prob:risk_cvar_meas} vectors are 
\begin{align}
        {\mathbf{c}} &= [0, 0, 0, 0, \textrm{id}_\R, 0], &
    {\mathbf{b}} &= [0, 1, h, 1, 0],
\end{align}

forming the variable pairings
\begin{align}
    \inp{\mathbf{c}}{\bbmu} &= \inp{\textrm{id}_\R}{\nu}, & \inp{\mathbf{b}}{\bell} &= \gamma + h \xi + \beta.
\end{align}

Under these definitions for $(\mathcal{A}, \mathbf{b}, \mathbf{c}, \mathcal{X}_+, \mathcal{Y}_+') $, Problems \ref{prob:risk_cvar_meas} and \ref{prob:risk_cvar_cont} can be placed in to standard forms \eqref{eq:risk_mean_cont_std} and \eqref{eq:risk_mean_cont_std} respectively (with objectives $p^*_{ES}$ and $d^*_{ES}$). This proves weak duality.

\subsection{Strong Duality}
Strong duality holds by extension of the arguments made in Section \ref{app:strong_duality_mean_sec}. Boundedness of measures for the \ac{ES} case is proven in the following lemma:
\begin{lem}
    All measure solutions to \eqref{eq:risk_cvar_meas} from Problem \ref{prob:risk_cvar_meas} are bounded under assumptions A1-A4. \label{lem:cvar_bounded}
\end{lem}
\begin{proof}
    Boundedness of $(\mu_0, \mu_\tau, \mu_+, \mu_-)$ was proven by Lemma \ref{lem:mean_mass}. The supports of $(\nu, \hat{\nu})$ are constrained to be compact in the set $[p_{\min}, p_{\max}]$ by definition in \eqref{eq:risk_meas_def_cvar}. The mass of $\nu$ is constrained to 1 by \eqref{eq:risk_cvar_meas_cvarmass}. Given that $\inp{1}{\mu_+} = \inp{1}{p_\# \pi_\#^x \mu_+} = h$, it holds by  \eqref{eq:risk_cvar_meas_cvarcon} that $\inp{1}{\hat{\nu}} = 1-\epsilon$. Both of $(\nu, \hat{\nu})$ are nonnegative measures with a finite support and a nonnegative mass, which proves that each of them are bounded. The lemma is therefore proven.
\end{proof}

Imposition of assumption A1 and satisfaction of Lemma \ref{lem:cvar_bounded} ensures that requirements R1 and R2 for strong duality are met. 
Continuity in R3 is maintained from the mean case, given that the identity map $\textrm{id}_\R$ is continuous.
The construction procedure used in the proof of Theorem \ref{thm:no_relaxation_cvar} satisfies R4. Because all requirements are met, strong duality in the \ac{ES} case between Problems \ref{prob:risk_cvar_meas} and \ref{prob:risk_cvar_cont} is proven.
\section{Proof of Convergence from Theorem \ref{thm:convergence_mean}}
\label{app:convergence_mean}
% \urg{TODO: This content will be excised from the CDC version, and will stay in the TAC/Arxiv versions only. No bulk-text rewriting will be necessary.}
In order to prove convergence of $p^*_k \rightarrow p^*$  and $p^*_{ES, k} \rightarrow p^*_{ES}$ as $k \rightarrow \infty$, the following lemma establishing boundedness of solutions for Problem %\ref{prob:risk_mean_meas} 
\ref{prob:risk_mean_mom} is required:

\begin{lem}
\label{lem:mean_mass_mom}
    At any level $k\geq 1$ of the hierarchy, all vectors $y^\bullet$ that are feasible for \eqref{eq:risk_mean_mom} or \eqref{eq:risk_cvar_mom} will have uniformly bounded first element $y^\bullet_0$ under A1-A4.
\end{lem}
\begin{proof}
    %Boundedness will be proven by the sufficient conditions of compact support and finite mass. Assumption A1 enforces compact support of all measures in \eqref{eq:risk_meas_def}. 
    \begin{subequations}

This proof will first cover the mean-type case in Problem \ref{prob:risk_mean_mom}, and then will advance to the \ac{ES}-type case in Problem \ref{prob:risk_cvar_mom}.
    
    By constraints \eqref{eq:risk_mean_mom_prob} and \eqref{eq:risk_mean_mom_window}, the initial and forward moment vectors $y^0$ and $y^+$ have first element
    \begin{equation} \label{eq:+_proof}
        y_0^0 = 1 < \infty, \qquad y^+_0 = h < \infty.
    \end{equation}
     Evaluating constraint \eqref{eq:risk_mean_mom_flow} at $\alpha = \beta = \gamma = 0$ results in \begin{equation} \label{eq:tau_proof}
         y_0^\tau = y_0^0 \stackrel{\eqref{eq:+_proof}}{=} 1 < \infty.
     \end{equation}
     It remains to uniformly bound $y^-_0$. Evaluating constraint \eqref{eq:risk_mean_mom_flow} at $\alpha = \gamma = 0$, $\beta = 1$ yields 
     \begin{equation} \label{eq:liouville_proof}
         y^\tau_{1,0} = y^+_0 + y^-_0 \stackrel{\eqref{eq:+_proof}}{=} h + y^-_0.
     \end{equation}
     From constraint~\eqref{eq:risk_mean_mom_lmi}, it holds that
     \begin{equation} \label{eq:-_proof}
         0 \leq M_0(y^-) = y^-_0,
     \end{equation}
     \begin{align}
         0 \leq \ & M_1(g_h y^\tau) \notag \\
         = \ & -y^\tau_{2,0} + (T+h)y^\tau_{1,0} - Thy^\tau_0 \notag \\
        \stackrel{\eqref{eq:tau_proof}}{=} & -y^\tau_{2,0} + (T+h)y^\tau_{1,0} - Th \label{eq:loc_proof}
     \end{align}
     and 
     $$0 \preceq M_1(y^\tau) = \begin{pmatrix}
             y^\tau_{0} & y^\tau_{1,0} & y^\tau_{0,1} \\
             y^\tau_{1,0} & y^\tau_{2,0} & y^\tau_{1,1} \\
             y^\tau_{0,1} & y^\tau_{1,1} & y^\tau_{0,2}
         \end{pmatrix}$$
         so that
     \begin{align}
         0 \leq \ & \begin{pmatrix}\nicefrac{(1+T+h)}{2} & -1 & 0\end{pmatrix} M_1(y^\tau) \begin{pmatrix}\nicefrac{(1+T+h)}{2} \\ -1 \\ 0\end{pmatrix} \notag \\
         = \ & y^\tau_0 \nicefrac{(1+T+h)^2}{4} - (1+T+h)y^\tau_{1,0} + y^\tau_{2,0} \notag \\
         \stackrel{\eqref{eq:tau_proof}}{=} & \nicefrac{(1+T+h)^2}{4} - (1+T+h)y^\tau_{1,0} + y^\tau_{2,0}. \label{eq:mom_proof}
     \end{align}
     Summing~\eqref{eq:loc_proof} and \eqref{eq:mom_proof} then yields
     $$0 \leq \nicefrac{(1+T+h)^2}{4} - Th -y^\tau_{1,0}$$
     which we reinject in~\eqref{eq:liouville_proof} to obtain, after rearrangement:
     \begin{equation} \label{eq:end_proof}
         0 \stackrel{\eqref{eq:-_proof}}{\leq} y^-_0 \leq \nicefrac{(1+T+h)^2}{4} - Th - h
     \end{equation}
     % which concludes the proof 
     Since the right hand side of \eqref{eq:end_proof} is indeed nonnegative due to the AM-GM inequality applied to $h$ and $T+1$, the proof is concluded for the mean case.

     Constraint \eqref{eq:risk_cvar_mom_prob}  relates at $\ell=0$ that $y^0_0 = \epsilon y^\nu_0 + y^{\hat{\nu}}_0$. Given that $y^\nu_0=1$ (constraint \eqref{eq:risk_mean_mom_prob}) and $y^0_0=1$ (constraint \eqref{eq:risk_cvar_mom_prob}), it holds that $y^{\hat{\nu}}_0 = 1-\epsilon$. Bounded first elements are now proven for \ac{ES} case as well.
    \end{subequations}
\end{proof}

We are now ready to prove Theorem \ref{thm:convergence_mean}. Convergence of $\lim_{k\rightarrow \infty} p^*_k = p^*$ and $\lim_{k\rightarrow \infty} p^*_{ES, k} = p^*_{ES}$ is assured by Corollary 8 of \cite{tacchi2022convergence}, given that all problem data is polynomial under an Archimedean requirement (A5-A6), all solutions are uniformly bounded  (Lemma \ref{lem:mean_mass_mom}), and the objective $p^*$ is finite (upper-bounded by the finite $\max_{x \in X} p(x)$ by compactness [A1] and polynomial structure [A5-A6]). 

\end{document}